\documentclass[12pt]{amsart}

\usepackage{a4wide}
\usepackage{float}
\usepackage{euscript,eufrak,verbatim}
\usepackage{graphicx}
\usepackage[usenames]{color}
\usepackage[colorlinks,linkcolor=red,anchorcolor=blue,citecolor=blue]{hyperref}
\usepackage{amsmath}
\usepackage{amsthm}
\usepackage[all]{xy}
\usepackage{graphicx} 
\usepackage{amssymb} 

\usepackage{bbold}

\newtheorem*{theomain}{Theorem}

\newtheorem{thm}{Theorem}[section]
\newtheorem{prop}[thm]{Proposition}
\newtheorem{df}[thm]{Definition}
\newtheorem{lem}[thm]{Lemma}
\newtheorem{cor}[thm]{Corollary}

\newtheorem{claim}[thm]{Claim}

\newtheorem{ques}[thm]{Question}

\newtheorem{conj}[thm]{Conjecture}
\newtheorem{rem}[thm]{Remark}

\def\N{\mathbb{N}}
\def\Z{\mathbb{Z}}
\def\N{\mathbb{N}}
\def\R{\mathbb{R}}
\def\C{\mathbb{C}}

\def\FF{\mathcal{F}}

\def\LL{\mathcal{L}}
\def\MM{\mathcal{M}}

\def\FFi{\overline{\mathcal{F}}}
\def\SS{\mathcal{S}}

\def\mdim{\text{\rm mdim}}

\def\Id{\text{\rm Id}}
\def\diam{\text{\rm diam}}

\def\BB{\mathcal{B}}

\def\TT{\mathcal{T}}
\def\Vect{\text{\rm Vect}}
\def\Ext{\text{\rm Ext}}
\def\M{\text{\rm Mult}}

\makeatletter 
\@addtoreset{equation}{section}
\numberwithin{equation}{section}

\newcommand{\norm}[1]{\left\lVert#1\right\rVert}

\makeatletter
\@namedef{subjclassname@2020}{\textup{2020} Mathematics Subject Classification}
\makeatother

\title{Multiplicity of topological systems}

	\author{David Burguet}
	\address
	{Sorbonne Universite, LPSM, 75005 Paris, France}
	\email{david.burguet@upmc.fr}
	\author{Ruxi Shi}
\address
	{Sorbonne Universite, LPSM, 75005 Paris, France}
\email{ruxi.shi@upmc.fr}

\subjclass[2020]{}
\keywords{}
\begin{document}

	\maketitle
	
	\begin{abstract}
We define the topological multiplicity of an invertible  topological system $(X,T)$ as the minimal number $k$ of real continuous functions  $f_1,\cdots, f_k$ such that the functions $f_i\circ T^n$, $n\in\mathbb Z$, $1\leq i\leq k,$ span a dense linear vector space in the space of real continuous functions  on $X$ endowed with the supremum norm. We study some properties of topological systems with  finite multiplicity. After giving some examples, we investigate the multiplicity of subshifts with linear growth complexity.
 
	\end{abstract}

\tableofcontents

\section{Introduction}

The multiplicity of an invertible bounded operator $U:E\circlearrowleft$ on a normed vector  space $E$  is the minimal cardinality of subsets $ F\subset E$, whose cyclic space (i.e. the vector space spanned by $U^kx$, $k\in \mathbb Z$, $x\in F$) is dense in  $E$. 

For  an ergodic measure preserving system $(X,f,\mathcal B,\mu)$,  the multiplicity $\M(\mu)$ of the Koopman operator, which is the operator of composition by $f$ on the Hilbert space $L^2(\mu)$ is a dynamical invariant, which has been investigated in many works (see e.g. \cite{dan} and the references therein). 

Cyclicity, which corresponds to simple multiplicity (i.e. there is an element whose cyclic space is dense in the whole vector space), has been also  established for operators of composition on the  Hardy space $H^2(D)$ \cite{shapiro}. In this context a  pioneering work of Birkhoff \cite{Birkh} states  that there is an entire function $\phi$ in the complex plane such that the set $\{\phi(\cdot+n), \ n\in \mathbb N\}$  is dense itself in the set of entire functions endowed with the   uniform topology on compact subsets, i.e. the operator of translation by $1$ is \textit{hypercyclic}.

 Quite surprisingly the corresponding topological invariant has not been studied in full generality.  More precisely we consider here  topological dynamical systems $(X,T)$, where $X$ is a compact metrizable space and $T:X\circlearrowleft$ is a homeomorphism and we study  the operator of composition by $T$ on the Banach space $C(X)$ of real continuous functions endowed with the uniform topology. We call topological  multiplicity  of $(X,T)$ the associated multiplicity  and we denote it by $\M(T)$. We remark that our definitions and results can be extended to the noninvertible continuous map $T:X\circlearrowleft$. But for sake of simplicity, we focus on  homeomorphisms $T:X\circlearrowleft$.
	
In this paper, we mostly focus on  topological systems with finite multiplicity. We first show the following properties for such systems.
\begin{theomain}
    Let $(X,T)$ be a topological system with finite multiplicity. Then the following properties are satisfied.
    \begin{enumerate}
        \item $(X,T)$ has zero topological entropy.
        \item $(X,T)$ has finitely many ergodic measures. 
    \end{enumerate}
\end{theomain}

These properties are the main contents of Section 2. The property (1) is proven in Proposition \ref{ent} in two ways: one  uses the  variational principal of topological entropy; the other is  purely topological. The property (2) is proven in Lemma \ref{erg} and Corollary \ref{err}. In fact we show more precisely that the number of ergodic measures is equal to the multiplicity of the operator induced on the quotient of $C(X)$ by the closure of coboundaries.

In Section 3, we relate the topological multiplicity with the dimension of cubical shifts, in which the action $T_*:\mathcal M(X)\circlearrowleft$ induced by $T$ on the set $\mathcal M(X)$ of Borel probability measures on $X$ may be affinely embedded. In Theorem \ref{embmul}, we show a necessary and sufficient condition for the existence of affinely embedding of $(\mathcal M(X), T_*)$ to the shift on $([0,1]^d)^\Z$. Furthermore, we compare our result to Lindenstrauss-Tsukamoto conjecture for dynamical embedding (Corollary \ref{cor:LT}).

In Section 4, we state a generalized Banach version of a lemma due to Baxter \cite{bax} which is  a classical criterion of  simplicity  for ergodic transformations. The generalized Baxter's Lemma (Lemma \ref{cyclic}) will play an important role on estimating the topological multiplicity in next sections.

For minimal Cantor systems  a topological  analogue of the rank of a measure preserving system  has been defined and  studied (see \cite{durand}). In Section 5, under this setting, we compare the topological multiplicity with the topological rank (Theorem \ref{thm rank}).

In Section 6, we study some examples and estimate their topological multiplicity: minimal rotations on compact groups, Sturmian and substitution subshifts, homeomorphisms of the interval, etc.

In Section 7, we estimate the topological multiplicity of  subshifts with linear growth complexity, i.e. subshifts $X$ such that the cardinality $p_X(n)$ of $n$-words in $X$ satisfies $\liminf_{n\to \infty}\frac{p_X(n)}{n}<+\infty$. Such subshifts  aroused a great deal of interest, specially recently \cite{boshernitzan1992condition,kracy,creutz2023low,donoso2021interplay}.  In \cite{boshernitzan1992condition} it is proved that an aperiodic subshift $X$ has  at most $k$ ergodic measures if  $\liminf_n\frac{p_X(n)}{n}\leq k\in \mathbb N$. Our main related result  states as follows (Theorem \ref{mubo} and Theorem \ref{last}):

\begin{theomain}
Let $X$ be an aperiodic subshift with $\liminf_{n\to \infty}\frac{p_X(n)}{n}\leq k\in \mathbb N$. Then 
$$\M(T)\leq 2k \text{ and }\sum_{\mu \ ergodic} \M(\mu)\leq 2k.$$
\end{theomain}

Except the results on multiplicity that we investigate, we propose several questions in the current paper.

\section{Topological multiplicity, definition and first properties}

\subsection{Multiplicity of a linear operator}
Let $(E,\|\cdot \|)$ be a normed vector  space over $\mathbb R$. 
 We consider a  linear invertible bounded operator $U:E\circlearrowleft$.  A subset $F$ of $E$ is called a \textit{generating family} of $U$ when    the vector space  spanned by $U^kx$, $k\in \mathbb Z$, $x\in F$, is dense in $E$. In the following we denote by $\text{span}(G)$ (resp. $\overline{\text{span}}(G)$)   the vector space spanned by a subset $G$ of $E$ (resp. its closure) and we  then let  $V^U_F:=\overline{\text{span}}\{ U^k x:  k \in \Z, x\in F \}$. Sometimes, we write $V_F$ instead of $V^U_F$ whenever the operator is fixed.  The \textit{multiplicity} $\M(U)\in \mathbb N\cup \{\infty\}$ of $U$ is then the smallest cardinality of generating families of $U$. By convention we let $\M(U)=0$ when $E$ is reduced to $\{0\}$.  A linear operator with multiplicity one is called \textit{cyclic}.

We first study the equivariant map between two normed vector spaces with linear invertible bounded operators.

\begin{lem}\label{ooo}
Let $U_i:E_i\circlearrowleft$, $i=1,2$ be two linear invertible bounded operators. Assume that there is a linear bounded operator $W:E_1\rightarrow E_2$  satisfying $W\circ U_1=U_2\circ W$ then 
$$\M(U_2|_{\overline{\mathrm{Im}(W)}})\leq \M(U_1).$$
\end{lem}
\begin{proof}
One checks easily that if $F$ is a generating family for $U_1$ then $W(F)$ is a generating family for the restriction of $U_2$ to the closure of the image of $W$. Therefore $\M(U_2|_{\overline{\mathrm{Im}(W)}})\leq \M(U_1)$.
\end{proof}

A direct consequence of Lemma \ref{ooo} is that the multiplicity is a spectral invariant : if $U_i$ are linear invertible operators on $E_i$  $,i=1,2$,    satisfying $W\circ U_1=U_2\circ W$ for some invertible bounded linear operator $W:E_1\rightarrow E_2$, then $U_1$ and $U_2$ have the same multiplicities.  

When $E'$ is a closed subspace of $E$ we endow the quotient $E/E'$ space with the norm $\|\overline{u}\|'=\inf\{\|u+v\|, \ v\in E' \}$. If $E'$ is invariant by $U$ we let $U^{E/E'}$  be the action induced by $U$ on the quotient normed space $E/E'$. In this context, by applying   Lemma \ref{ooo}  with $W:E\rightarrow E/E'$ being the natural projection, we get \begin{eqnarray}\label{encor}
\M(U^{E/E'})\leq \M(U).\end{eqnarray}

\subsection{Operator of composition : Topological and ergodic multiplicities.}
Ergodic theory focuses on the study of invertible measure preserving systems $(X,f,\mathcal B, \mu)$. In particular the spectral properties of the unitary operator $U_f:L^2(\mu)\circlearrowleft$,  $\phi\mapsto \phi\circ f$, are investigated.  We let $\|f\|_{2}:=(\int_X|f(x)|^2 d\mu)^{1/2}$ be the $L^2$-norm of $f\in L^2(\mu)$.

\begin{df}
The \textbf{ergodic multiplicity} $\M(\mu)$ of an ergodic system $(X,f,\mathcal B, \mu)$ is the multiplicity of the restriction of $U_f$ to the Hilbert space $L_0^2(\mu):=\{f\in L^2(\mu), \ \int f\, d\mu=0\}$, that is to say, $\M(\mu)=\M(U_f)$.
\end{df}
 This  quantity has been intensely studied in ergodic theory (see Danilenko's survey \cite{dan}).

Next we consider here  an invertible topological dynamical system $(X,T)$, i.e. $T:X\circlearrowleft$ is a homeomorphism of a compact metric space $X$. We denote by $C(X)$ the Banach space of  real continuous functions endowed with the topology of uniform convergence. We let $\|f\|_{\infty}:=\sup_{x\in X}|f(x)|$ be the supremum norm of $f\in C(X)$.
\begin{df}
The \textbf{topological multiplicity} $\M(T)$ of $(X,T)$ is the multiplicity of 
the operator of composition $U_T:C(X)\circlearrowleft$, $\phi\mapsto \phi\circ T$.
\end{df}
 Quite surprisingly this last notion seems to be new (note however that cyclicity of $U_T$ has already been investigated in some cases). Let us first observe that the topological multiplicity  bounds from above  the ergodic multiplicity of ergodic $T$-invariant measures.


 
\begin{lem}\label{compa}
Let $(X,T)$ be an invertible topological dynamical system.  For any ergodic $T$-invariant measure $\mu$, we have
$$\M(\mu)\leq \M(T).$$
\end{lem}

\begin{proof}
Let $F$ be a generating family with minimal cardinality of $U_T:C(X)\circlearrowleft$. Then the vector space spanned by $F$ is dense in $(C(X),\|\cdot\|_\infty)$, therefore in $(L^2(\mu),\|\cdot\|_2)$. As $p:L^2(\mu)\rightarrow L^2_0(\mu)$, $f\mapsto f-\int f\ d\mu$ is continuous and $p\circ U_T=U_T\circ p$, the vector space spanned by $p(F)$ is dense in $L^2_0(\mu)$.

\end{proof}
Let $\mathcal M(X)$ be the set of Borel probability measures endowed with the weak-$*$ topology. It is standard that  $\mathcal M(X)$ is  a compact metrizable space.
The compact  subset $\mathcal M(X,T)\subset \mathcal M(X)$ of Borel $T$-invariant probability measures of $(X,T)$ is  a simplex, whose extreme set  is given by the subset $\mathcal M_e(X,T)$ of ergodic measures. A topological system with a  unique (ergodic) invariant measure is said to be {\it uniquely ergodic}.  
Jewett-Krieger theorem states that every ergodic system has a uniquely ergodic model. Several proofs have been given of this theorem, e.g. see \cite[Section 29]{denker2006ergodic}. One may wonder if the multiplicity may be preserved:

\begin{ques}
Given an ergodic system with measure $\mu$, is there a uniquely ergodic model $(X,T)$ of it such that $\M(T)=\M(\mu)$?
\end{ques}

\subsection{The  number of ergodic  measures as a multiplicity.}
Let $(X,T)$ be an invertible topological dynamical system. 
A function $\psi\in C(X)$ is a called a {\it continuous $T$-coboundary}, if $\psi$ is equal to $\phi\circ T-\phi$ for some $\phi\in C(X)$. In other terms the  set $B_T(X)$ of continuous $T$-coboundaries is the image of $U_T-\Id$, in particular it is a vector space. Observe that $U_T(B_T(X))= B_T(X)$. To simplify the notations we write  $\tilde{U}_T$ for the action induced by $U_T$ on the quotient Banach space $C(X)/ \overline{B_T(X)}$ and $\underline{U}_T$ for the restriction of $U_T$ to the closure $\overline{B_T(X)}$ of continuous coboundaries.
By a standard application of Hahn-Banach theorem (see e.g. Proposition 2.13 in \cite{Katt}), a function $\psi$ belongs  to $\overline{B_T(X)}$ if and only if $\int \psi \, d\mu=0$ for any $\mu\in \mathcal M(X,T)$ (resp. $\mu\in \mathcal M_e(X,T)$). 
It is well-known that unique ergodicity is equivalent to the decomposition $C(X)=\mathbb R \mathbb  1\oplus \overline{B_T(X)}$ (see e.g. Lemma 1 in \cite{vol}), where $\mathbb  1$ denotes the constant function equal to $1$.  In particular in the case of unique ergodicity,  we have $C(X)/ \overline{B_T(X)}\simeq \mathbb R \mathbb 1$ and therefore $\M(\tilde U_T)=1$. 
It may be generalized as follows. 

\begin{lem}\label{erg}
Let $(X,T)$ be an invertible topological dynamical system. We have
$$\M(\tilde U_T)=\sharp \mathcal M_e(X,T).$$
\end{lem}

\begin{proof}
We first show that $\M(\tilde U_T)\geq \sharp \mathcal M_e(X,T)$. 
Assume that :
\begin{itemize}
\item $\nu_1,\cdots, \nu_p$ are distinct ergodic measures,
\item $\overline{F}=\{\overline{f_1}, \cdots, \overline{f_q}\}\in C(X)/ \overline{B_T(X)} $ is a generating family of $\tilde{U}_T$. 
\end{itemize}
For $1\le l\le q$, let $f_l\in C(X)$ be a function (a priori not unique)  such that $\overline{f_l} = f_l \mod \overline{B_T(X)}$.
If $q<p$ then the $p$ vectors  $$ X_i=\left(\int f_l\, d\nu_i\right)_{l=1, 2, \cdots, q}, i=1, 2,\cdots, p$$ are linearly dependent in $\R^q$, i.e. there is $(c_i)_{1\le i\le p}\in \mathbb R^{p} \setminus (0,0, \dots, 0)$ such that \begin{align}\label{fini}\sum_{1\le i\le p}c_iX_i	&=0.\end{align} Let $\nu $ be the signed measure $\nu=\sum_{1\le i\le p}c_i\nu_i$.  Then Equality (\ref{fini}) may be rewritten as follows: 
$$
    \forall 1\le l\le q, \ \int f_l \, d\nu=0.
$$
The measures $\nu_i$ being invariant for ${1\le i\le p}$, so is $\nu$. Therefore we get   
\begin{equation}\label{eq:tt}
\forall 1\le l\le q, \forall k\in \Z,  \ \int f_l\circ T^k \, d\nu=0.
\end{equation}
But $V^{\tilde{U}_T}_{\overline{F}}=C(X)/ \overline{B_T(X)}$,  so that  for any $\epsilon>0$ and for any $g\in C(X)$, we may find  $h \in \mathrm{span}(f_l\circ T^k, \ 1\le l\le q,k\in \Z)$ and $u\in B_T(X)$  with $\|g-(h+u)\|_\infty<\epsilon$.  By \ref{eq:tt} we have $\int h\, d\nu=0$. As $u$ is a coboundary, we have also $\int u\, d\nu=0$. Therefore 
\begin{align*}\left|\int g\, d\nu\right|&\leq \left|\int (h+u)\, d\nu\right|+\|g-(h+u)\|_\infty<\epsilon.
\end{align*} Since $\epsilon>0$ and $g\in C(X)$ are chosen arbitrarily, we obtain $\int g\, d\nu=0$,  for any $g\in C(X)$, therefore $\nu=0$. This contradicts the ergodicity of the measures $\nu_i$ for $1\le i\le p$. Consequently we have $q\geq p$ and therefore $\M(\tilde U_T)\geq \sharp \mathcal M_e(X,T)$.

Let us show now the converse inequality. Without loss of generality we may assume that $p=\sharp \mathcal M_e(X,T)< \M(\tilde{U}_T)=q<\infty$. We let again:
 \begin{itemize}
\item $\mathcal M_e(X,T)=\{\nu_1,\cdots, \nu_p\}$,
\item $\overline{F}={\overline{f_1}, \cdots, \overline{f_q}}\in C(X)/ \overline{B_T(X)} $ a generating family of $\tilde{U}_T$ with minimal cardinality. 
\end{itemize}
Then the $q$ vectors $$Y_l=\left(\int f_l\,d\nu_i\right)_{i=1,\cdots, p}, \ l=1,\cdots, q,$$ are linearly dependent in $\R^p$, i.e. there is $(c_l)_{1\le l\le q}\in \mathbb R^{q}\setminus (0,0 \dots, 0)$ such that $$\sum_{1\le l\le q}c_lY_l=0.$$ Let  $g$ be the function $g=\sum_{1\le l\le q}c_lf_l$. Then  we have  $$\int g \,d\nu_i=0, \ \forall 1\le i\le p,$$  
A previously  mentioned, it implies that $g$ lies in $\overline{B_T(X)}$.
This contradicts the minimality of the generating family $\overline{F}$. 
\end{proof}

\begin{rem}\label{r}
It follows from the proof of Lemma \ref{erg} that if $ \mathcal M_e(X,T)=\{\nu_1,\cdots, \nu_p\}$, then   ${\overline{f_1}, \cdots, \overline{f_p}}$  is a  generating family  of $\tilde{U}_T$ if and only if the matrix $A=(\int f_j\,d\nu_i)_{1\le i,j\le p}\in M_p(\mathbb R)$ is invertible. 
\end{rem}

By inequality (\ref{encor}) and Lemma \ref{erg} we get:
\begin{cor}\label{err}
$$\sharp \mathcal M_e(X,T)\leq \M( T).$$
\end{cor}

\subsection{Relating $\M(T)$ and $\M(\underline{U}_T)$}
It follows from definition of $\overline{B_T(X)}$ that the map $$W:C(X)\rightarrow \overline{B_T(X)}, \ f\mapsto f\circ T-f$$ has dense image  and commutes with $U_T$.
By applying  Lemma \ref{ooo} with $U_1=U_2=U$ and $E_1=C(X)$, $E_2=\overline{B_T(X)}$, we obtain  $\M(\underline{U}_T)\leq \M(T)$. 

We show then in this subsection the following inequality.
\begin{prop}\label{mu}
$$\M(T)\leq \M(\tilde{U_T})+\M(\underline{U}_T)-1.$$
\end{prop}

In particular if $(X,T)$ is uniquely ergodic, $\M(T)=\M(\underline{U}_T)$ by Lemma \ref{erg}. Let us now prove Proposition \ref{mu}. For a family $F$ of $C(X)$, we write $\overline{F}$ the subset of $C(X)/\overline{B_T(X)}$ consisting of $\overline{f}= f \mod \overline{B_T(X)}$ for  $f\in F$. We start with a technical lemma.

\begin{lem}\label{lem:-1}
	Let $(X,T)$  be an invertible dynamical system with $\sharp \mathcal M_e(X,T)< \infty$. If $F$ is a  family of $C(X)$ such that $\overline F$ is generating   for $\tilde{U}_T$, then the constant function $\mathbb 1$ belongs to $V_F$. 
\end{lem}

\begin{proof}
	Let $\mathcal M_e(X, T)=\{\nu_1,\cdots,\nu_p\}$.  By Remark \ref{r} the matrix $(\int f_i\,d\nu_j)_{1\le i,j\le p}$ is invertible. Then by replacing  $F=\{f_1,\cdots, f_p\}$ by some invertible linear combinations we can assume 
	$$
	\forall 1\leq i,j\leq p, \ \int f_i\,d\nu_j=\delta_{i,j},
	$$
 where $\delta_{i,j}$ is equal to $1$ if $i=j$ and $0$ otherwise.
 Let $f=\sum_{i=1}^p f_i$. We have 
	\begin{align}\label{eq:nu}
	\forall 1\leq j\leq p, &\int f \, d\nu_j=\sum_i\delta_{i,j}=1,\nonumber\\
	\text{therefore, } \forall \nu\in \mathcal{M}(X, T),&
		\int f \, d\nu=1.
	\end{align}
	 We claim that $\frac{1}{N} \sum_{n=0}^{N-1} f\circ T^n$ is converging uniformly to $\mathbb 1$ as $N$ goes to infinity. If not, there would exist a positive number $\epsilon$, a sequence $(x_k)_{k\ge 1}$ and an increasing sequence $(N_k)_{k\ge 1}$ of positive integers such that  
	\begin{equation}\label{eq:epsilon}
		\left| \frac{1}{N_k} \sum_{n=0}^{N_k-1} f(T^n(x_k)) -1 \right|>\epsilon, \ \forall k\ge 1.
	\end{equation}
	After passing to a subsequence of $(N_k)_{k\ge 1}$, we might assume that $\frac{1}{N_k} \sum_{n=0}^{N_k-1} \delta_{T^n(x_k)}$ is  converging to a $T$-invariant measure $\mu$ in the weak-$*$ topology. It follows from \eqref{eq:epsilon} that 
 $$
 \left| \int f d\mu -1 \right|>\epsilon>0.
 $$
 It is a contradiction to \eqref{eq:nu}. Therefore, $\frac{1}{N} \sum_{n=0}^{N-1} f\circ T^n$ is converging uniformly to $\mathbb 1$ as $N$ goes to infinity, in particular $\mathbb 1\in V_F$.
\end{proof}



\begin{proof}[Proof of Proposition \ref{mu}] Let $\overline{\mathcal F}=\{\overline{f_1}, \cdots,   \overline{f_p}\}$  and $\mathcal G=\{g_1, \cdots, g_q\}$ be generating families  of $\tilde{U}_T$ and $\underline{U}_T$ with $p=\M(\tilde{U}_T)$ and $q=\M(\underline{U}_T)$.  For $1\le l\le p$, take $f_l\in C(X)$ be a function such that $\overline{f_l} = f_l \mod \overline{B_T(X)}$, then let $\mathcal F=\{f_1,\cdots, f_l\}$. One easily checks that $\mathcal F\cup \mathcal G$ is a generating family of $U_T$.  By Lemma \ref{erg} we may write $\mathcal M_e(X, T)=\{\nu_1,\cdots,\nu_p\}$. As in the proof of Lemma \ref{lem:-1}  we may assume without loss of generality $\int f_i\, d\nu_j=\delta_{i,j}$ for any $1\leq i,j\leq p$.   Let $g'_1=g_1+\mathbb  1$, hence $$\left(\int g'_1\, d\nu_i \right)_{1\le i\le p}=(1,\cdots, 1).$$ By Remark \ref{r} the family  $\{ \overline{g'_1}, \overline{f_j}: \ 1<j\le p \}$ is generating for $\tilde{U}_T$. By Lemma \ref{lem:-1} the constant functions, therefore also $g_1$, belongs to $V_{ \{ g'_1, f_j : \   1<j\le p  \} }$. Then $V_{\{ g'_1, f_j,g_i \ : \  1<j\le p, \, 1<i\le q \} }=V_{\{ g'_1, f_j,g_i \ : \  1<j\le p, \, 1\le i\le q \} }$ and $f_1\in V_{\{ g'_1, f_j,g_i \ : \  1<j\le p, \, 1\le i\le q \} }$.
Consequently we get
\begin{align*}
V_{\{ g'_1, f_j,g_i \ : \  1<j\le p, \, 1<i\le q \} }&\supset V_{\mathcal F\cup \mathcal G }= C(X).
\end{align*}
We conclude that $\M(T)\leq p+q-1=\M(\tilde{U_T})+\M(\underline{U}_T)-1.$
\end{proof}

\section{Affine embedding of $(\mathcal M(X), T_*)$ in cubical shifts}

For a topological system $(X,T)$ we denote by $T_*$ the action induced by $T$  on the compact set $\mathcal M(X)$, i.e.  $T_*\mu(\cdot)=\mu(T^{-1}\cdot)$ for all $\mu\in \mathcal M(X)$. Then $(\mathcal M(X), T_*)$ is also a topological system, which is called the {\it induced system} of $(X,T)$.

For  $d\in \mathbb N$ we let $\sigma_d$ be the shift on the simplex $([0,1]^d)^\mathbb Z$. An {\it embedding} of $(X,T)$ in $([0,1]^d)^\mathbb Z$ is a continuous injective map $\phi:X\rightarrow ([0,1]^d)^\mathbb Z$ satisfying $\phi\circ T= \sigma_d \circ \phi$.  Existence of such embedding is related to the mean dimension theory (we refer to \cite{coornaert2015topological} for an introduction). Such an embedding implies that the mean dimension of $(X,T)$ is less than or equal to $d$. Moreover the topological dimension (i.e. Lebesgue covering dimension) $d_n^T$ of the set of $n$-periodic points then also satisfy $\frac{d_n^T}{n}\leq d$.  Conversely it has been shown that minimal 
systems with mean dimension less than $d/2$ can be embedded in the cubical shift $\sigma_d$ \cite{L99,gutman2020embedding}. \\

In this section we consider affine embedding of the induced system $(\mathcal M(X), T_*)$ in cubical shift $\sigma_d$, i.e. the embedding   $\phi:\mathcal M(X)\rightarrow ([0,1]^d)^\mathbb Z$ is affine. In particular we will relate the embedding dimension $d$ with the multiplicity of $(X,T)$.

\subsection{Case of finite sets}
We first deal with the case of a finite set $X$.  Then $T$ is just a permutation of $X$ and $\mathcal M(X)$ is a finite dimensional simplex. We classify the possible affine embedding of $(\mathcal M(X),T_*)$ in the following proposition.

\begin{prop}\label{dimf}
Suppose $X$ is a finite set and $T$ is a a permutation of $X$. Let $\tau_1\cdots \tau_k$ be the decomposition of $T$ into disjoint cycles $\tau_i$ of length $r_i$ for $1\le i\le k$.
\begin{itemize}
    \item[(1)] If there is a nontrivial common factor of $r_i$ for $1\le i\le k$, then there is an affine embedding of $(\mathcal M(X),T_*)$ in $\left( ([0,1]^k)^{\mathbb Z}, \sigma_k\right)$. Such $k$ is sharp.
    \item[(2)] If there is no nontrivial common factor of $r_i$ for $1\le i\le k$, then there is an affine embedding of $(\mathcal M(X),T_*)$ in $\left( ([0,1]^{k-1})^{\mathbb Z}, \sigma_{k-1}\right)$. Such $k-1$ is sharp.
\end{itemize}
\end{prop}

\begin{proof}
 For each $1\le i\le k$ we fix a point $e_i\in X$ in each cycle $\tau_i$, i.e. $\{T^je_i: 0\le j\le r_i \}=X$. Notice that there are continuous maps $a_e:\mathcal M(X)\rightarrow [0,1]$ , $e\in X$, with $\sum_{e\in X} a_e=\mathbb 1$ satisfying
$\mu=\sum_{e\in X}a_e(\mu)\delta_e$ for all $\mu\in \mathcal M(X)$.  \\
(1) Assume there is a nontrivial common factor $p$ of $r_i$ for $1\le i\le k$. Then $$\mathrm{dim}(\mathrm{Fix}(T_*^p))=kp-1,$$
where $\mathrm{Fix}(T_*^p)=\{\mu\in \mathcal M(X): \ T_*^p\mu=\mu\}$. Since $p>1$ and $\mathrm{dim}(\mathrm{Fix}(\sigma_{k-1}^p))=kp-p$, the dynamical system $(\mathcal M(X),T_*)$ can not embed in $\left( ([0,1]^{k-1})^{\mathbb Z}, \sigma_{k-1}\right)$.  

Now we construct the embedding of $(\mathcal M(X),T_*)$ in $\left( ([0,1]^k)^{\mathbb Z}, \sigma_k\right)$. We define firstly a dynamical embedding $\Psi$ of the set of extreme points in $\mathcal M(X)$, which is identified with $X$ through the map $x\mapsto \delta_x$, into $([0,1]^k)^\mathbb Z$ by letting 
$$ \forall i=1,\cdots ,k , \ \forall l\in \mathbb Z, \  \ (\Psi(T^{l}e_i))_i=\sigma^l \left( (10^{r_i-1})^\infty\right);$$
the other components $(\Psi(T^{l}e_i))_j$, $j\neq i$, being chosen to be equal to the $0^\infty$ sequence.  Then we may extend $\Psi$ affinely from the set of extreme points  on $\mathcal M(X)$ by letting 
$$\Psi(\mu)=\sum_{e\in X} a_e(\mu)\Psi(\delta_e).$$
It is easy to check that $\Psi$ is injective which deduces a dynamical embedding $(\mathcal M(X),T_*)$ in $\left( ([0,1]^k)^{\mathbb Z}, \sigma_k\right)$.

(2) Assume there is no nontrivial common factor of $r_i$ for $1\le i\le k$. We have that $k-1$ numbers $q_{i}:=(r_k, r_i),$ $1\le i\le k-1$ are co-prime  where $(a,b)$ are the highest common factor of $a$ and $b$. We define firstly a continuous map $\Psi$ of the set of extreme points in $\mathcal M(X)$ into $([0,1]^{k-1})^\mathbb Z$ by letting 
$$
\ \forall l\in \mathbb Z, \  \ (\Psi(T^{l}e_k))_{j}=\sigma^l \left( (10^{r_i-1})^\infty\right), \forall 1\le j\le k-1,
$$
and
$$ \forall 1\le i\le k-1 , \ \forall l\in \mathbb Z, \  \ (\Psi(T^{l}e_i))_{i}=\sigma^l \left( (10^{r_i-1})^\infty\right); \ \forall j\not=i, \forall l\in \mathbb Z, (\Psi(T^{l}e_i))_j=0^\infty.$$
Then we may extend $\Psi$ affinely from the set of extreme points  on $\mathcal M(X)$ by letting 
$$\Psi(\mu)=\sum_{e\in X} a_e(\mu)\Psi(\delta_e).$$
It remains to show that $\Psi$ is injective. Let $\mu=\sum_{e\in X} b_e \delta_e$ and $\mu'=\sum_{e\in X} b_e' \delta_e$. Suppose $$\Psi(\mu)=\Psi(\mu')=\sum_{e\in X} a_e\Psi(\delta_e)=(c_{i,j})_{1\le i\le k-1, j\in \Z}.$$ 
Since $q_{i}:=(r_k, r_i)$ then there are integers $s_i$ and $t_i$ such that 
$
s_ir_i-t_ir_k=q_i.
$
Let 
$$
u_{i,l}=s_ir_i+l=t_ir_k+q_i+l.
$$
It implies that 
$$
b_{T^le_k}'-b_{T^{l+q_i}e_k}'=b_{T^le_k}-b_{T^{l+q_i}e_k}=c_{i,l}-c_{i,u_{i,l}}.
$$
Since $q_i, 1\le i\le k-1$ are co-prime, there are integers $w_i, 1\le i\le k-1$ such that $\sum_{i=1}^{k-1}w_iq_i=1$. Since
\begin{align*}
    \left(b_{e_k}-b_{T^{w_1q_1}e_k} \right)+&\left(b_{T^{w_1q_1}e_k}-b_{T^{w_1q_1+w_2q_2}e_k} \right) \\ &+\dots+\left(b_{T^{w_1q_1+w_2q_2+\dots+w_{k-2}q_{k-2}}e_k}-b_{T^{w_1q_1+w_2q_2+\dots+w_{k-1}q_{k-1}}e_k} \right)=b_{e_k}-b_{Te_k},
\end{align*}
we have 
$$
b_{e_k}-b_{T^le_k}=b_{e_k}'-b_{T^le_k}', \ \forall l\in \Z.
$$
Since $\sum_{e\in X}b_e=\sum_{e\in X}b_e'=1$, we conclude that $b_{e_k}=b_{e_k'}$ and consequently $b_{e_i}=b_{e_i'}$ by $b_{e_k}+b_{e_i}=b_{e_k}'+b_{e_i}'=c_{i,0}$ for $1\le i\le k-1$. It means that $\mu=\mu'$ and $\Psi$ is injective.

\end{proof}

\begin{rem} For such a permutation $T$, we have $\M(T)=\sharp \mathcal M_e(X,T)=k$, with $k$ being the number of cycles  in the decomposition of $T$.
\end{rem}

\subsubsection{General case}

We consider now a general topological system and relates the dimension of the cubical shift in an affine embedding with the multiplicity of $(X,T)$.

\begin{thm}\label{embmul}
Let $(X,T)$ be a topological system. If $\M(T)$ is equal to $d$, then there is an affine embedding of $(\mathcal M(X), T_*)$ in $(([0,1]^d)^\mathbb Z,\sigma_d)$.  Conversely if 
$(\mathcal M(X), T_*)$ embeds into $(([0,1]^d)^\mathbb Z,\sigma_d)$ then 
\begin{itemize}
\item either $\sharp \mathcal M_e(X,T)\leq d$ and $\M(T)\leq d$,
\item or $\sharp \mathcal M_e(X,T)=d+1$ and $\M(T)=d+1$.
\end{itemize}
\end{thm}
\begin{proof}
	
Firstly, notice that any affine equivariant map $\Psi:(\mathcal M(X), T_*)\rightarrow (([0,1]^d)^\mathbb Z,\sigma_d)$ is of the form $$\Psi_f:\mu\mapsto \left(\int f\circ T^k\, d\mu\right)_{k\in \Z},$$ for some continuous function $f=(f_1,\cdots,f_d):X \rightarrow [0,1]^d$. 
	
Assume the topological multiplicity $\M(X,T)$ is equal to $d$, i.e. there is a family $F=\{f_1,\cdots,f_d\}$ of continuous  functions  such that $V_F=C(X)$.  Let us show the associated map $\Psi_f$ is injective. Let $\mu_1, \mu_2\in \mathcal M(X)$ with $\Psi_f(\mu_1)=\Psi_f(\mu_1)$ i.e. $\int f_i \circ T^k \,d\mu_1=\int f_i \circ T^k \, d\mu_1$ for any $i=1, \cdots, d$ and any $k\in \mathbb Z$. Then by density of $\mathrm{span}( f_i \circ T^k, \ i,\ k)$ in $C(X)$  we have  $$\int g \,d\mu_1=\int g\, d\mu_1 \text{ for all }  g\in C(X),$$ which implies $\mu_1=\mu_2$. Therefore we get the injectivity of $\Psi_f$.

Conversely, assume $\Psi_f$ is injective for $f=(f_1,\cdots,f_d):X \rightarrow [0,1]^d$. Let $F=\{\mathbb 1, f_1,\cdots,f_d \}$. We claim that $V_F=C(X)$. Then if $\sharp \mathcal M_e(X,T)\leq d$,   we get by injectivity of $\Psi_f$ that there exists $A\subset \{1,2, \dots, d \}$ with $\sharp A=\sharp  \mathcal M_e(X,T)$ such that the matrix $\left( \int f_i d\nu\right)_{i\in A, \nu\in  \mathcal M_e(X,T)}$ is invertible. Then by Remark \ref{r} the family $\overline{F\setminus\{\mathbb 1 \}}$ is generating for $\tilde{U}_T$  and consequently $V_F=V_{F\setminus\{\mathbb 1 \}}$ by Lemma \ref{lem:-1}. If $\sharp  \mathcal  M_e(X,T)= d+1$ then we only get $\M(T)=d+1$.

 It remains to show our claim. Assume to the contrary that  $V_F\not=C(X)$. Then by Riesz Theorem there is  a signed finite measure $\mu$  vanishing on each function in $F$. Let  $\mu=\mu^+-\mu^-$ be the Jordan decomposition of $\mu$ (i.e. the measures $\mu^+$ and $\mu^-$ are two finite positive measures which are mutually singular).  Evaluating on the constant function $\mathbb 1$, we get $\mu^+(X)=\mu^-(X)$. Then by rescaling, we may assume both $\mu^-$ and $\mu^+$  belong to $\mathcal M(X)$. Finally we get $\Psi_f(\mu^+)=\Psi_f(\mu^-)$, and therefore $\mu^+=\mu^-$ by injectivity of $\Psi_f$ contradicting therefore the mutual singularity of  $\mu^+ $ and $\mu^-$.

\end{proof}


\subsection{Affine embeddings and Lindenstrauss-Tsukamoto conjecture}
Lindenstrauss and Tsukamoto \cite{lindenstraussTsukamoto2014mean} have conjectured that any topological system with mean dimension $\mdim(X,T)$ less than $d/2$ and such that the dimension $d_n^T$ of the set of $n$-periodic points satisfies $\frac{d_n^T}{n}< d/2$
for any $n\in \mathbb N$ may be embedded in the shift over $([0,1]^d)^\mathbb Z$. As mentioned above it is known for minimal systems. We consider here affine systems,  i.e. affine maps of a simplex. 
Such maps are never minimal, as they always admits at least  one  fixed point.\\

The example below shows that Lindenstrauss-Tsukamoto conjecture does not hold true in the affine category.  Recall that an ergodic system $(X,f,\mathcal B,\mu)$  has a \textit{countable Lebesgue spectrum}, when  there is a countable family $(\psi_n)_{n\in \mathbb N}$ in $L^2_0(\mu)$ such that $\psi_n\circ f^k$, $k\in \mathbb Z$, $n\in \mathbb N$ form a Hilbert basis of 
$L^2_0(\mu)$.

\begin{cor}\label{cor:LT}
There is an affine system  with a unique periodic (fixed) point and zero-topological entropy (in particular $mdim(T)=0$ and $d_n^T=0$ for all $n\ge 1$) which does not embed affinely in $(([0,1]^k)^\mathbb Z,\sigma)$ for any $k\ge 1$.
\end{cor}

\begin{proof}
There exists an ergodic measure preserving system $(Y,\mathcal A,f,\mu)$ with zero entropy and countable Lebesgue spectrum \cite{spec, hor} (in particular \text{\it totally ergodic}, i.e. $f^n$ is ergodic for any $n\in \Z$). Then by Jewett-Krieger theorem there is a uniquely ergodic topological system $(X,T)$ with measure $\nu$  realizing such a measure preserving system. All powers of $T$ are uniquely ergodic as $\mu$ was chosen totally ergodic.   Moreover the topological entropy of $T$, thus that of $T_*$ is zero by Glasner-Weiss \cite{glasner1995quasi}. As the unique invariant measure $\nu$ has countable Lebesgue spectrum, the topological multiplicity of $(X,T)$ is infinite by Lemma \ref{compa}.  We conclude with Theorem  \ref{embmul}.
\end{proof}

\subsection{Application: zero topological entropy}

A classical result in ergodic theory  states that any ergodic system $(X,f,\mathcal B,\mu)$ with positive entropy has a countable Lebesgue spectrum. In particular $[h(\mu)>0]\Rightarrow [\M(\mu)=\infty]$. Then it follows from the variational principle for the topological entropy :

\begin{prop}\label{ent}
Any topological system $(X,T)$ with $\M(T)<\infty$ has zero topological entropy.
\end{prop}

We may also give a purely topological proof of Proposition \ref{ent} based on mean dimension theory. More precisely we use the main result of \cite{burguet2022topological}, which  states as follows: 

\begin{thm}\cite{burguet2022topological}\label{ru}
For any topological system $(X,T)$ with positive topological entropy, the induced system $(\mathcal M(X), T_*)$ has infinite topological mean dimension. Therefore, $$h_{top}(T)>0\Leftrightarrow \mdim(T_*)>0 \Leftrightarrow \mdim(T_*)=\infty.$$
\end{thm}
\begin{proof}[Topological proof of Proposition \ref{ent}]
Assume $\M(T)=d$ is finite. Then by Theorem \ref{embmul} the induced system  $(\mathcal M(X), T_*)$ embeds in the cubical shift $\left(([0,1]^d)^\mathbb{Z},\sigma\right)$. In particular the mean dimension of $T_*$ is less than or equal to the mean dimension of the shift $\left(([0,1]^d)^\mathbb{Z},\sigma\right)$, which is equal to $d$. By Theorem \ref{ru}, it implies that $T$ has zero topological entropy. 
\end{proof}

\section{Baxter's Lemma in Banach spaces}

In \cite{bax}, Baxter gave  a useful criterion to show simple spectrum of ergodic transformations. It may be extended more generally  to bound the multiplicity of the spectrum (e.g. see Proposition 2.12 in \cite{queffelec2010substitution}). We generalize this criterion for operators defined on a Banach space. It will be used in the next section to estimate the topological multiplicity in some examples.   

\begin{lem}\label{cyclic}
	Let $B$ be a separable Banach space and $\mathcal L(B)$ be the set of bounded linear operator on $B$.  We consider  an invertible isometry  $U\in\mathcal L(B)$. If $(\mathcal {F}_n)_n$ is a sequence of finite subsets in $H$ satisfying for all $f\in B$ 
	\begin{eqnarray}\label{baxtou} \inf_{F_n\in V_{\mathcal F_n}}\|F_n-f\|\xrightarrow{n\rightarrow \infty} 0
	\end{eqnarray}
	
	 then there exists a family $\mathcal {F}\subset B$ with $\sharp\mathcal {F} \le \sup_n \sharp \mathcal {F}_n$ and
	$B=V_\mathcal {F}.$
	
\end{lem}
Classical proofs of Baxter's lemma strongly used the Hilbert structure. Here we use a Baire argument as in Lemma 5.2.10 \cite{fer}. Note also that we do not require the sequence of vectors spaces 
$(V_{\mathcal F_n})_n$ to be nondecreasing. Observe finally that it is enough to assume (\ref{baxtou}) for $f$ in $S$, where $S$ spans a dense subset of $B$.

\begin{proof}
	Let $m=\sup_n \sharp \mathcal {F}_n$. If $m=\infty$, it is trivial. Assume $m<\infty$. By passing to a subsequence, we assume $\sharp \mathcal {F}_n=m$ for all $n$. Let $B^{(m)}$ be the space of finite subsets of $H$ whose cardinality is smaller than or equal to $m$. When endowed with  the Hausdorff distance $d_{Hau}$, the space $B^{(m)}$ is a metric space, which is complete and separable.  We assume the following claim, which we prove later on. 
\begin{claim}\label{cl}
 For any $\epsilon>0$ and $\mathcal{F}\in B^{(m)}$, the set $$O(\mathcal F,\epsilon)=\left\{\mathcal{G}\in B^{(m)},  \  \forall f\in \mathcal{F} \  \inf_{G\in V_{\mathcal{G}}} \|G-f\|<\epsilon \right\}$$ is open and  dense.
\end{claim}

Let $\left(g_q\right)_{q\in \mathbb N}$ be a countable dense family  in $B$.	Let $\mathcal G_q$ be the finite family $\{g_1,\cdots, g_q\}$.  For any $p\in \mathbb N^*$ and any  $q\in \mathbb N $ we consider the open and dense set 
	$$O_{p,q}:= O(\mathcal G_q,1/p)= \left\{\mathcal{G}\in B^{(m)},  \  \forall g\in \mathcal{G}_q \  \inf_{G\in V_{\mathcal{G}}} \|G-g\|<1/p \right\}.$$
	According to Baire's theorem, the intersection $\bigcap_{p,q}O_{p,q}$ is not empty. Clearly any family $\mathcal F$ in the intersection satisfies $V_{\mathcal F}=B$. It remains to prove Claim \ref{cl}. 
	
	\begin{proof}[Proof of Claim \ref{cl}] The set  
$O(\mathcal F,\epsilon)$ is  open. We focus on the denseness property. Pick arbitrary $\delta>0$ and  $\mathcal{H}\in B^{(m)}$. We will show that there is  $\mathcal H'\in O(\mathcal F,\epsilon)$ with $d_{Hau}(\mathcal H,\mathcal H')<\delta$. As the elements of $B^{(m)}$ with cardinality $m$ are dense in $B^{(m)}$ we can assume without loss of generality that $\sharp \mathcal H=m$.  By assumptions on the sequence $\mathcal F_n$, there exists $n$ such that 
	\begin{align}\label{easy}
	 \forall f\in \mathcal{F}\  \inf_{F_n\in V_{\mathcal F_n}} \|F_n-f\|<\epsilon,
\end{align}
\begin{align}\label{fff}
	 \forall h\in  \mathcal H \  \inf_{F_n\in V_{\mathcal F_n}} \|F_n-h\|<\delta.
\end{align}
	We write $\mathcal{H}=\{h_1, h_2, \dots, h_m \}$ and $\mathcal{F}_n=\{f_1, f_2, \dots, f_m \}$. 
By (\ref{fff}) there are polynomials  $(P_{i,j})_{1\le i,j\le m}$ in $\mathbb R[X]$ and a nonnegative  integer $p$ such that	
\begin{eqnarray}\label{ouf}\forall i=1,\cdots,m, \ \left\|h_i-\sum_{j=1}^m U^{-p} P_{i,j}(U)f_j\right\|<\delta\end{eqnarray}
Let $Q\in \mathbb R[X]$ be the polynomial given by the determinant of the matrix $M=(P_{i,j})_{1\le i,j\le m}\in M_m(\mathbb R[X])$. The spectrum $\mathrm{Sp}(U)$ of $U$ is contained in the unit circle. In particular for arbitrarily small $\lambda\in \mathbb R$, the polynomial $Q(\cdot+\lambda)$ does not vanish on $\mathrm{Sp}(U)$. Hence by replacing $P_{i,j}$ by $P_{i,j}(\cdot +\lambda)$ we may assume that $Q$ does not vanish on the spectrum of $U$. Then $Q(U)=\prod_{\lambda, \, Q(\lambda)=0}(U-\lambda \mathrm{Id})$ is invertible and its inverse may be approximated by polynomials in $U$  and $U^{-1}$, because for $\lambda$ with $Q(\lambda)=0$ we have $(U-\lambda \mathrm{Id})^{-1}=-\sum_{k\in \mathbb N}\frac{U^k}{\lambda^{k+1}}$ for $|\lambda|>1$ and $(U-\lambda \mathrm{Id})^{-1}=-\sum_{k\in \mathbb N}\frac{U^{-(k+1)}}{\lambda^{k}}$ for $|\lambda|<1$ (these sequences are normally convergent in $\mathcal L(B)$ as we assume $\|U\|=\|U^{-1}\|=1$).
Let $\mathcal H'=\{h'_1, h'_2, \dots, h'_m \}$ with $h'_i=\sum_{j=1}^m U^{-p}P_{i,j}(U)f_j$. We have \begin{align*}\mathcal F_n&= M(U)^{-1}U^p\mathcal H',\\
&=\mathrm{^{t}com} \, M(U)U^pQ(U)^{-1}\mathcal H'.
\end{align*}
Then from the above observations we get $$\mathcal F_n\subset V_{\mathcal H'}.$$ In particular  $V_{\mathcal F_n}\subset V_{\mathcal H'}$ and  it follows finally from (\ref{easy}) that $\mathcal H'\in O(\mathcal F,\epsilon)$, i.e. \begin{align*}
	 \forall f\in \mathcal{F}\  \inf_{H' \in V_{\mathcal H'}} \|H'-f\|<\epsilon.
\end{align*}

This completes the proof 
 as we have $d_{Hau}(\mathcal H,\mathcal H')<\delta$  by (\ref{ouf}), where $\mathcal H$ and $\delta$ have been chosen arbitrarily. 
\end{proof}
\end{proof}

\section{Cantor systems with finite topological rank}
Roughly speaking an ergodic measure preserving system is of finite rank $r$, when it may be obtained by cutting and staking with $r$ Kakutani-Rohlin towers. 
For ergodic systems, Baxter's lemma implies that the ergodic multiplicity is less than or equal to the rank. 
Topological rank as been defined and studied for minimal Cantor  systems (see e.g. \cite{donoso2021interplay} and the references therein).  For such systems we show now with Lemma \ref{cyclic} that the same inequality holds for the topological quantities : the topological multiplicity is less than or equal to the topological rank.

Firstly we recall the definition of topological rank. Let $(X,T)$ be a minimal Cantor system. A {\it Kakutani-Rohlin partition} of $X$ is given by
$$
\TT=\{T^{-j} B(k): 1\le k\le d, 0\le j<h(k) \}, 
$$
where $d, h(k), 1\le k\le d$ are positive integers and $B(k), 1\le k\le d$ are clopen subsets
of $X$ such that
$$
\cup_{k=1}^{d} T^{-h(k)}B(k)=\cup_{k=1}^{d} B(k).
$$
The {\it base} of $\TT$ is the set $B(\TT)=\cup_{k=1}^{d} B(k)$. A sequence of Kakutani-Rohlin partitions
$$
\TT_n=\{T^{-j} B_n(k): 1\le k\le d_n, 0\le j<h_n(k) \}, n\ge 1,
$$
is {\it nested} if 
\begin{enumerate}
	\item $\TT_0$ is the trivial partition, i.e. $d_0=1, h_0=1$ and $B_0(1)=X$.
	\item $B(\TT_{n+1})\subset B(\TT_n)$.
	\item $\TT_{n+1}\succ \TT_n$.
	\item $\sharp(\cap_{n\ge 0} B(\TT_n))=1$.
	\item $\cup_{n\ge 1} \TT_n$ spans the topology of $X$.
\end{enumerate}
Moreover, it is {\it primitive} if for all $n\ge 1$ there exists $N > n$ such that for all $1 \le k \le d_N$ and for each $x\in T^{-(hN (k)-1)}B_N(k)$,
$$
\{T^i(x): 0\le i\le h_N(k)-1   \} \cap B_n(j)\not=\emptyset, \forall 1\le j\le d_n.
$$
Following \cite{donoso2021interplay}, a minimal Cantor system is of {\it topological rank} $d$ if it admits a primitive sequence of nested Kakutani-Rohlin partitions with $d_n\le d$ for all $n\in \N$.

\begin{thm}\label{thm rank}
	Let $(X,T)$ be a minimal Cantor system with topological rank $d$. Then $\M(X,T)\le d$.
\end{thm}
\begin{proof}
	Let $(\TT_n)_{n\in \N}$ be the primitive sequence of nested Kakutani-Rohlin partitions with $d_n\le d$ for all $n\in \N$. Let 
	$$
	\mathcal{F}_n=\{\chi_{B_n(k)}: 1\le k\le d_n \}.
	$$
	Since $\cup_{n\ge 1} \TT_n$ spans the topology of $X$,  we have 
	$$\forall f\in C(X), \ \inf_{F_n\in V_{\mathcal F_n}}\|F_n-f\|\xrightarrow{n\rightarrow \infty} 0.$$
	It follows from Lemma \ref{cyclic} and $d_n\le d$ for all $n\in \N$ that $\M(X,T)\le d$.
\end{proof}

\begin{rem}
	It was shown in \cite{donoso2021interplay} that the Thue-Morse subshift has topological rank $3$. By Theorem \ref{thm rank},  Thue-Morse subshift has therefore topological multiplicity at most $3$. We give another direct proof of this fact in Subsection \ref{subs}.
\end{rem}

Examples of ergodic systems with rank $r$ and multiplicity $m$ have been built for any $1\leq m\leq r$  in \cite{Lacroix}. We then propose the following question.
\begin{ques}
Can one build for any $1\leq m\leq r$ a minimal Cantor system with topological multiplicity $m$ and topological rank $r$?
\end{ques}

\section{Examples of finite topological multiplicity} 
An invertible dynamical system is called {\it topological simple} or {\it have simple topological spectrum} if $\M(T)=1$.
\subsection{Minimal rotation on compact groups}
Let $G$ be a compact abelian group. Denote by $\hat G $ the dual group  of $G$ and by $\lambda$ the Haar measure on $G$. For $f\in C(G)$, we write $\hat{f}$ the Fourier transformation of $f$. 
\begin{prop}\label{trans}
	Any minimal translation $\tau$ on a compact abelian group $G$ is topologically simple.
\end{prop}

\begin{proof}
	We claim that  any $f\in C(G)$ with $\hat f(\chi)\neq 0$ for all $\chi\in \hat G$ is cyclic, i.e. the  vector space  spanned by 
	$f\circ \tau ^k$, $k\in \mathbb N$ is dense in $C(X)$. As characters of a compact abelian group separates points it is enough to show by Stone-Weierstrass theorem that any character belongs to the complex vector space spanned by $f\circ \tau ^k$, $k\in \mathbb N$. But for all $\chi\in \hat G $ we have 
	$$\hat f(\chi) \chi=\chi * f=\int f(\cdot -y)\chi(y) \, d\lambda(y).$$
	Then the function $f$ being uniformly continuous, there are functions of the form $\sum_k f(\cdot-y_k)\chi(y_k)$ arbitrarily close to $\hat f(\chi) \chi$ for the supremum norm.  
	By   minimality of $\tau$, there are integers $l_k\in \mathbb N$ such that $f\circ \tau^{l_k}$ and $f(\cdot-y_k)$ are arbitrarily closed. It concludes the proof.
\end{proof}

\subsection{Sturmian subshift}
 A word $u \in \{0, 1\}^{\mathbb Z}$ is called {\it Sturmian} if it is recurrent under the shift $\sigma$, and the number of $n$-words in $u$ equals $ n + 1$ for each $n \geq  1$. Take the shift-orbit closure $X_u =\overline{ O_\sigma(u)}$. The corresponding subshift $(X_u,\sigma)$ is called a {\it Sturmian subshift}. Sturmian sequences are symbolic representation  of circle irrational rotations.
 
We first recall some standard notations in symbolic dynamics.  For a subset $Y$ of $\mathcal A^{\mathbb Z}$ with $\mathcal A$ being  a finite alphabet  we let $\mathcal L_n(Y)$ be the number of $n$-words appearing in the sequences of $Y$. Then for 
$w\in \mathcal L_n(Y)$ we  let $[w]$ be the associated cylinder defined as 
$[w]:=\{(x_n)_{n\in \Z}\in Y: \ x_0\cdots x_{n-1}=w\}.$  The indicator function of a subset $E$ of $X$ will be denoted by $\chi_E$.
 
\begin{prop}\label{Strurmian}
	Any Sturmian subshift has simple topological spectrum.
\end{prop}
\begin{proof}
	Let $u$ be a Sturmian sequence. Let
	$$
	F_n=\text{span}\{\chi_{[w]}: w\in\mathcal{L}_n(u)\}.
	$$
	It follows that
	$$
	C(X_u)=V_{\cup_n F_n}.
	$$
	Notice $\dim(F_n)=\sharp \mathcal{L}_n(u) =n+1$. We let $f:X_u\rightarrow \mathbb R$ be the continuous function defined as   $f:x=(x_n)_n\mapsto(-1)^{x_0}$. Let
	$$
	G_n=\text{span}\{\mathbb{1}, f\circ \sigma^k: 0\le k \le n-1 \}.
	$$
	Clearly, $G_n\subset F_n$. To prove that $(X_u, \sigma)$ has simple topological spectrum, it is sufficient to show $\dim(G_n)=n+1$. Thus it is enough to show the functions $\{ \R \mathbb 1, f\circ \sigma^k: 0\le k \le n-1\}$ are linearly  independent. If not, for some $n$ there exists a nonzero vector $(a_0, a_1, \dots, a_n)$ such that
	$$
	a_0(-1)^{x_0}+a_1(-1)^{x_1}+\dots+a_{n-1}(-1)^{x_{n-1}}+a_n=0, \forall x\in X_u.
	$$
	Since $\sharp \mathcal L_{n-1}(u)> \sharp \mathcal L_{n-2}(u)$, we can find distinct $x, x'\in X_u$ such that $x|_0^{n-2}=x'|_0^{n-2}$ but $x_{n-1}\not=x'_{n-1}$. It follows that $a_{n-1}=0$. Since for each $0\le k\le n-2$ we can always find $y, y'$ such that $y|_0^{k-1}=y'|_0^{k-1}$ but $y_{k}\not=y'_{k}$, we obtain that $a_{n-2}=a_{n-3}=\dots=a_{0}=0$. Finally, we get $a_n=0$. This is a contradiction. Therefore we conclude that $\dim(G_n)=n+1$, then $F_n=G_n$. Then by Lemma \ref{cyclic} and Lemma \ref{lem:-1}, $(X_u, \sigma)$ has simple topological spectrum.
\end{proof}

\subsection{Homeomorphism of the interval}

Estimating the multiplicity of non-zero dimensional systems is  difficult in general. Below we focus on homeomorphisms of the interval (see \cite{jav} for related results on the circle). 
We use the following result due to Atzmon and Olevskii \cite{atz}. We denote by $C_0(\mathbb R)$ the set of continuous map  on $\mathbb R$ with zero limits in $\pm \infty$. For $f\in 
C_0(\mathbb R)$ and $n\in \mathbb Z$ we let $f_n=f(\cdot+n)$ be the translation of $f$ by $n$.
\begin{thm}\label{atzz}\cite{atz}
There exists $g\in C_0(\mathbb R)$ such that the vector space spanned by $g_n$, $n\in \mathbb N$ is dense in $ C_0(\mathbb R)$. 
\end{thm}
In particular the operator $V: C_0(\mathbb R)\circlearrowleft$, $f\mapsto f(\cdot+1) $, is cyclic.  A Borel set $S$ of $\mathbb R$ is called a {\it set of uniqueness} if 
the sets $S_n:=(S+2\pi n)\cap [-\pi,\pi]$, $n\in \mathbb Z$ satisfy the following properties:
\begin{enumerate}
\item $S_n$, $n\in \mathbb Z$, are pairwise disjoint, 
\item  $\mathrm{Leb}(S_n\cap U)>0$ for any $n\in \Z$ and any open set $U$ of $[-\pi, \pi]$,
\item   $\mathrm{Leb}(S)<\infty$,
\end{enumerate}
where $\mathrm{Leb}$ denotes the Lebesgue measure on $\R$.

 Atzmon and Olevskii proved for any set of uniqueness $S$  (such sets exist!)  the conclusion of Theorem \ref{atzz} holds true with $g$ being the 
 the  Fourier transform of the indicator function of $S$.  Let us just remark that if $S$ is a set of uniqueness  then $$S^l=\bigcup_{n\in \mathbb Z} (S_{nk+l}+2\pi k), 1\leq l\leq k,$$ are $k$ disjoints sets of  uniqueness. Let $C_0(\mathbb R;\mathbb C)$ be the set of continuous map  on $\mathbb C$ with zero limits in infinity.
\begin{lem}\label{Atzmon}
The operator  $V: C_0(\mathbb R;\mathbb C)^{k}\circlearrowleft$, $(f_i)_{1\leq i\leq k}\mapsto \left(f_i(\cdot+1)\right)_{1\leq i\leq k}$ is cyclic. In particular, the operator  $U: C_0(\mathbb R)^{k}\circlearrowleft$, $(f_i)_{1\leq i\leq k}\mapsto \left(f_i(\cdot+1)\right)_{1\leq i\leq k}$ is cyclic.
\end{lem}

\begin{proof}
Let $S$, $S^l$, $1\leq l\leq k$,  be sets of uniqueness as above. By following \cite{atz} we show that the vector space generated by the translates of 
 $g:=(\widehat{\chi_{S^l}})_{1\leq l\leq k}$  is dense in $C_0(\mathbb R;\mathbb C)^{k}$ with $\widehat{\chi_{S^l}} 
 $ be the Fourier transform of the indicator function $\chi_{S^l}$ of $S^l$. It follows that the translates of 
 $Re(g)$  is dense in $C_0(\mathbb R)^{k}$. Let $\mu=(\mu_l)_{1\leq l\leq k}$ 
 be a complex bounded measure with $$\langle V^n(g), \mu \rangle=\sum_{1\leq l\leq k} \int (\widehat{\chi_{S^l}})_n d\mu_l=0,$$ for all $n\in 
 \mathbb Z$. It is enough to prove $\mu_l=0$ for all $1\leq l\leq k$. By Plancherel-Parseval formula we have $$\int 
 (\widehat{\chi_{S^l}})_n d\mu_l= \int \widehat{(\widehat{\chi_{S^l}})_n}(t)\hat{\mu_l}(-t)  \,dt= -
 \int \chi_{S^l}(t) e^{-int}\hat{\mu_l}(t) \,dt.$$ Therefore we have
 $$\sum_{1\leq l\leq k} \int \chi_{S^l}(t) \hat{\mu_l}(t)e^{-int}\, dt=0,$$ for all $n$. But this term is just the 
 $n^{th}$ coefficient of the function of $L^1([-\pi,\pi])$ given by $\sum_{m\in \mathbb Z}(\sum_{1\leq l\leq k}
 \chi_{S^l}\hat \mu_l)(\cdot +2\pi m) $, which should therefore be $0$.  As the sets $(S^l+2\pi m)\cap [-
 \pi,\pi], m\in \Z,$ are pairwise disjoint, each term of the previous sum should be zero ;  that is  $(\chi_{S^l}\hat \mu_l)(x +2\pi k)=0$ for all $m,l$ and for Lebesgue almost every $x\in [-\pi, \pi]$.  By Property (2) in the definition 
 of a set of uniqueness, we conclude $\widehat{\mu_l}=0$. Therefore $\mu_l=0$ for each $1\leq l\leq k$ and consequently the translates of 
 $g$  is dense in $C_0(\mathbb R;\mathbb C)^{k}$.  
\end{proof}

\begin{prop}
Let $f:[0,1]\circlearrowleft$ be a  homeomorphism of the interval. Then 
$$\M(f)=\sharp \mathcal M_e([0,1],f).$$
\end{prop}

\begin{proof}
We first deal with the case of an increasing homeomorphism. The ergodic measures of $f$ are the Dirac measures at these fixed points. Notice that $f$ has at least two fixed points, $0$ and $1$. If it has infinitely many fixed points, then $\M(U_f)\ge \sharp \mathcal M_e([0,1],f)=\infty$. Now assume it has finitely many fixed points.  Let $2\le k+1<+\infty$ be the number of fixed points.   Since $\varphi (x)-\varphi \circ f (x)=0$ for any continuous function $\varphi\in C(X)$ and any fixed point $x$, the space $\overline{B_f([0,1])}$ is the set of  real continuous maps on the interval which vanishes at the fixed points. It follows that the operator  $\underline{U}_f$ is spectrally conjugate to $V: C_0(\mathbb R)^{k}\circlearrowleft$, $(f_i)_{1\leq i\leq k}\mapsto \left(f_i(\cdot+1)\right)_{1\leq i\leq k}$. By Lemma \ref{Atzmon} we have $\M(\underline{U_f})=1$. It follows then from Proposition \ref{mu} and Proposition \ref{erg} that 
\begin{equation}\label{eq1234}
    \sharp \mathcal M_e([0,1],f)\leq \M(U_f)\leq \sharp \mathcal M_e([0,1],f)+\M(\underline{U_f})-1=\sharp\mathcal M_e([0,1],f).
\end{equation}

It remains to consider the case of a decreasing homeomorphism $f$. Let $0<a<1$ be the unique fixed point of $f$. Then $f^{2}:[0,a]\circlearrowleft$ is an increasing homeomorphism. Let $0=x_1<x_2<\cdots<x_k=a$ be the fixed points of $f^{2}|_{[0,a]}$. Then the ergodic measures of $f$ are the atomic periodic measures 
$\delta_a$ and  $\frac{1}{2}\left(\delta_{x_i}+\delta_{f(x_i)}\right)$
for $i=1, \cdots, k-1$. In particular we have $k=\sharp \mathcal M_e([0,1],f)$. From the previous case there is a generating family $\mathcal G=\{g_1, \cdots, g_k\}$ for $f^{2}:[0,a]\circlearrowleft$. Let $h\in C([0,1])$. For any $\epsilon>0$, there are $N\in \mathbb N$, $a_{l,n}$ and $b_{l,n}$, for $l=1,\cdots, k $ and $|n|\leq N$, (depending on $\epsilon$), such that 
\begin{align}
\|h-\sum_{l,n}a_{l,n}g_l\circ f^{2n}\|_{{[0,a]},\infty}&<\epsilon,\label{evla}
\end{align}
and
\begin{align}
\|h\circ f^{-1} -\sum_{l,n}b_{l,n}g_l\circ f^{2n}-h(a)\|_{{[0,a]},\infty}&<\epsilon,\label{evla1}
\end{align}
where $\|g\|_{{[0,a]},\infty}=\sup_{x\in [0,a]} |g(x)|$.
We consider the extension  $\tilde{g_l}$ of $g_l$ to $[0,1]$ with $\tilde{g_l}=g_l(a)$ on $[a,1]$.  We check now that  $\tilde{\mathcal G}=\{\tilde g_1, \cdots, \tilde g_k\}$ is generating for $f$. It follows from (\ref{evla}) and  (\ref{evla1}) at $x=a$ that 
\begin{equation}\label{ev3}
    \left|h(a)-\sum_{l,n}a_{l,n} g_l(a)\right|<\epsilon \text{ and } \left|\sum_{l,n}b_{l,n} g_l(a)\right|<\epsilon.
\end{equation}
Observe that 
\begin{align*}
h&=h|_{[0,a]}+(h\circ f^{-1}|_{[0,a]})\circ f|_{[a,1]}.
\end{align*}  
Combining \eqref{ev3} with (\ref{evla}), we obtain that
\begin{align*}
    &\|h-\sum_{l,n}a_{l,n}\tilde{g_l}\circ f^{2n}-\sum_{l,n}b_{l,n}\tilde{g_l}\circ f^{2n+1}\|_{[0,a], \infty} \\
    = & \|h-\sum_{l,n}a_{l,n}{g_l}\circ f^{2n}-\sum_{l,n}b_{l,n}{g_l}(a) \|_{[0,a] \infty} <2\epsilon.
\end{align*}
Similarly, combining \eqref{ev3} with (\ref{evla1}) we get 
\begin{align*}
    &\|h-\sum_{l,n}a_{l,n}\tilde{g_l}\circ f^{2n}-\sum_{l,n}b_{l,n}\tilde{g_l}\circ f^{2n+1}\|_{[a,1], \infty} \\
    = & \|h-\sum_{l,n}a_{l,n}{g_l}(a)-\sum_{l,n}b_{l,n}{g_l}\circ f^{2n+1} \|_{[a,1] \infty} \\
    \le & |h(a)-\sum_{l,n}a_{l,n}{g_l}(a)|+ \| h-\sum_{l,n}b_{l,n}{g_l}\circ f^{2n+1} -h(a) \|_{[a,1] \infty}
    <2\epsilon.
\end{align*}
Therefore we have 
$$\|h-\sum_{l,n}a_{l,n}\tilde{g_l}\circ f^{2n}-\sum_{l,n}b_{l,n}\tilde{g_l}\circ f^{2n+1}\|_\infty<2\epsilon.$$
We conclude that $\tilde{\mathcal G}$ is generating for $f$ as $\epsilon>0$ and $h\in C([0,1])$ have been chosen arbitrarily.

\end{proof}

\begin{ques}
What is the topological multiplicity of a Morse-Smale diffeomorphism?
\end{ques}

\subsection{Substitution Subshift}\label{subs}

For  a finite alphabet $\mathcal A$ with $\sharp \mathcal A\geq 2$ we let $\mathcal A^+$ be the set of associated finite words,  i.e.  finite  sequences of letters in $\mathcal A$.   For a word $w$ in $\mathcal A^+$ we let $|w| $ be the number of letters of $w$.  A  {\bf substitution} is a map $\zeta:\mathcal A\rightarrow \mathcal A^+$.   We say $\zeta$ has the {\bf constant length} if $|\zeta(a)|$ does not depend on $a\in \mathcal A$.   
The substitution may be extended to finite and infinite (one-sided) words by concatenation: $\zeta(u_0u_1\cdots)= \zeta(u_0)\zeta(u_1)\cdots$ for $u_i\in \mathcal{A}$ and $i\in \mathbb N$.  In this way $\zeta$ can be iterated.  

The langage $\mathcal L_\zeta$ of $\zeta$ is the set of  all subwords of $\zeta^n(a)$ over $a\in \mathcal A$ and $n\in \mathbb N$,  i.e.  the set of finite sequence of consecutive letters appearing in  the words  $\zeta^n(a)$.  The associated subshift $X_\zeta\subset \mathcal A^\mathbb{Z}$ is the set of  sequences $(x_n)_{n\in \mathbb Z}$ such that $x_px_{p+1}\cdots x_q$ belongs to $\mathcal L_\zeta$ for any $p\leq q\in \mathbb Z$. The map $\zeta$ may be extended to map from $X_\zeta$ to $X_\zeta$ by letting 
$$\zeta( x)=\cdots \zeta( x_{-1})\zeta(x_{-1})\, | \, \zeta(x_0)\cdots \zeta( x_{1})\cdots$$ for any $x=(x_n)_{n\in\Z}\in X_{\zeta}$, where  the non negative coordinates of $\zeta(x)$  lie after the symbol ``$|$".  For $w\in \mathcal A$ we recall that the associated cylinder $[w] $  in $X_\zeta$  is given by the sequences in $X_{\zeta}$ whose $|w|$ first coordinates coincide with $w$. 

A substitution $\zeta$ is said to be {\bf primitive} when there exists $n\in \mathbb N^*$ such that for  any letters $a,b$ the word  $\zeta^n(a)$ contains the letter $b$.  It is well-known that the shift $\sigma$ on $X_\zeta$ with primitive $\zeta$ is then minimal and uniquely ergodic (see \cite{queffelec2010substitution}).  In the followings, we only consider aperiodic primitive substitutions, i.e. primitive substitutions such that $X_\zeta$ is not reduced to a periodic orbit.  

We will make use of the standard recognazibility property of primitive substitutions: 

\begin{thm}\label{mosh}\cite{mosse}
Let $\zeta$ be an aperiodic primitive substitution. 
Then  for any positive integer $n$,  the family  $$P_n:=\{\sigma^k\left(\zeta^n[a]\right), \ a\in \mathcal A \ \text{ and } 0\leq k< |\zeta^n(a)|\}$$ defines a partition of $X_{\zeta}$ into clopen sets.  \\ 

Moreover $\zeta^n: X_{\zeta}\rightarrow \zeta^n(X_\zeta)$ is a topological conjugacy between 
$(X_\zeta,\sigma)$ and  the induced system  $\left(\zeta^n(X_\zeta),\sigma_{\zeta^n(X_\zeta)}\right)$ defined as  $\sigma_{\zeta^n(X_\zeta)}(x)=\sigma^{|\zeta^n(x_0)|}(x)$ for any $x=(x_k)_{k\in \Z}\in \zeta^n(X_{\zeta})$.
\end{thm}

We define two equivalence relations on $\mathcal A$ as follows:
$$a\sim^{d}b\text{ if and only if } \text{ $\zeta(a)$ and $\zeta(b)$ begin with the same letter};$$
$$a\sim^{f}b\text{ if and only if } \text{ $\zeta(a)$ and $\zeta(b)$ end with the same letter}.$$
We then let $\mathcal A_d$ and $\mathcal A_f$ be respectively the equivalence classes of $\sim^d$ and $\sim^f$.  
   We let $N_d$ and $N_f$ be the number of equivalence classes for $\sim^d$ and $\sim^f$.

\begin{thm}\label{multsub}Let $\zeta$ be a primitive aperiodic substitution (resp.  with constant length). Then the multiplicity 
$(X_\zeta, \sigma)$ is less than or equal to (resp. strictly less than) $$K_\zeta= \sharp \mathcal A(N_d+N_f)+\min(N_d,N_f)-(\sharp \mathcal A+N_dN_f),$$
which is less than or equal to $\sharp \mathcal A^2$.
\end{thm}

For a proper substitution (i.e.  when $N_d=N_f=1$) we get  $K_\zeta=\sharp \mathcal A$.  The partitions $(P_n)_{n\in \N}$ in Theorem \ref{mosh} are not refining, i.e. the diameter of $P_n$ does not go to zero with $n$ in general.  To prove Theorem \ref{multsub} we use the following finer clopen partitions $(Q_n)_{n\in \N}$.  For a subset $ J\subset \mathcal A$ and $a\in \mathcal A$ we let $[a| J]\subset X_{\zeta}$ (resp.  $[J |a]$) be the union of  $\sigma ([aj])$ (resp. $\sigma ([ja])$) over $j\in  J$).

\begin{lem}\label{bofff}
Let $\zeta$ be an aperiodic primitive substitution and let $M_n$ be the integer part of $\min_{a\in \mathcal A}|\zeta^{n-1}(a)|/2$.  Then  for any positive integer $n$,  the family  \begin{align*}Q_n:=&\{\sigma^l\left(\zeta^n[J_f| a]\right): \ a\in \mathcal A, \ J_f\in \mathcal A_f \text{ and } 0\leq l<M_n\} \bigcup \\
& \{\sigma^{-l}\left(\zeta^n[a| J_d]\right): \ a\in \mathcal A, \  J_d\in \mathcal A_d  \text{ and } 0\leq l<|\zeta^n(a)|-M_n\}
\end{align*}
defines a partition of $X_{\zeta}$ into clopen sets whose diameter is going to zero when $n$ goes to infinity.
\end{lem}

\begin{proof}
Clearly, for $0\leq k<M_n$, the clopen sets  $\sigma^k\left(\zeta^n[J_f| a]\right)$  for $J_f\in \mathcal A_f$  defines a partition of  $\sigma^{k}(\zeta^n[a])$. On the other hand, for $0\leq l<|\zeta^n(a)|-M_n$,  the clopen sets  $\sigma^{-l}\left(\zeta^n[a| J_d]\right)$  for $J_d\in \mathcal A_d$  defines a partition of  $\sigma^{|\zeta^n(a)|-l}(\zeta^n[a])$.  Therefore we deduce then from Theorem \ref{mosh} that $Q_n$ is a partition of $X_\zeta$.  It remains to show that their diameter is going to zero when $n$ goes to infinity. Fix $a\in \mathcal A$ and  $J_f\in \mathcal A_f$  and let $b\in \mathcal A$ such that $\zeta(j)$ ends with the letter $b$  for any $j\in J_f$.  Then any sequence in $\zeta^n[J_f|a]$ lies in $\sigma^{|\zeta^{n-1}(b)|}\left([\zeta^{n-1}(b)\zeta^n(a)]\right)$.  
In particular, for any $x,y\in \sigma^{l}\left(\zeta^n[J_f|a]\right)$ with $0\le l<M_n$, we 
have $x_{-M_n+1}\cdots x_{M_n-1}=y_{-
M_n+1}\cdots y_{M_n-1}$  since we have 
$M_n\leq \max(|\zeta^n(a)|/2, |\zeta^{n-1}(b)|)$.  
Arguing similarly with $\sigma^{-l}\left(\zeta^n[a| J_d]\right)$ we conclude that the diameter of $Q_n$ goes to zero when $n$ goes to infinity, because $M_n$ goes to infinity.
\end{proof}
For a finite collection $\mathcal C$ of clopen sets and for $n\in \mathbb N^*$ we let $\mathcal F_n^{\mathcal C}$ be the set of indicator functions $\chi_{\zeta^{n}(C)}$ with $C\in \mathcal C$. We also write 
$\mathcal F(Q_n)$ for the set of indicator functions $\chi_{A}$ with $A\in Q_n$.
\begin{lem}\label{ouch}
There is a finite collection $\mathcal C$ of clopen sets with $\sharp \mathcal C\leq K_\zeta$,  such that $$\forall n\in \mathbb N^*, \ V_{\mathcal F(Q_n)}=V_{\mathcal F_n^{\mathcal C}}.$$
\end{lem}

\begin{proof}
Without loss of generality we may assume $N_d\leq N_f$. Fix $\overline{J_f}\in \mathcal A_f$ and $\overline{i_d}\in J_d$ for any $J_d\in \mathcal A_d$. We consider the collection 
$$\mathcal C:=\left \{[a|J_d],\ [J_f|b]\ :\ a\in \mathcal A, \ J_d\in\mathcal A_d, \ J_f\in \mathcal A_f\setminus \{\overline{J_f}\}, \ b\in J_d\setminus \{\overline{i_d}\} \   \right\}.$$

Clearly $V_{\mathcal F(Q_n)}\supset V_{\mathcal F_n^{\mathcal C}}$ and   $$\sharp \mathcal C\leq N_d\sharp \mathcal A+ (N_f-1)(\sharp \mathcal A-N_d)=K_\zeta.$$

We show now the other inclusion $V_{\mathcal F(Q_n)}\subset V_{\mathcal F_n^{\mathcal C}}$.
Let $\tilde{\mathcal C}_n$ be the collection of clopen sets $E$ satisfying   $  \chi_{\zeta^n (E)}\in V_{\mathcal F_n^{\mathcal C}}$. 
It is enough to check $[a|J_d]$ and $[J_f|b]$ lie in $\tilde{\mathcal C}_n$ for any $a,b\in \mathcal A$, $J_d\in \mathcal A_d$ and $J_f\in \mathcal A_f$. 
 For  $E, E'\in \tilde{\mathcal C}_n$, the union $E\cup E'$ (resp. $E\setminus E'$) also belongs to $\tilde{\mathcal C}_n$ if $E\cap E'=\emptyset$ (resp. $E'\subset E$), because $ \chi_{\zeta^n(E\cup E')}=\chi_{\zeta^n(E')}+\chi_{\zeta^n(E)}$ (resp.  $ \chi_{\zeta^n(E\setminus E')}=\chi_{\zeta^n(E)}-\chi_{\zeta^n(E')}$ ) as $\zeta^n$ is one-to-one.

 Then for any   $J_d\in \mathcal A_d$ and any $J_f\in \mathcal A_f$, the set $[J_f|J_d]=\coprod_{a\in J_f}[a|J_d]$ lie in $\tilde{\mathcal C}_n$, because they  are  disjoint unions of elements of $\mathcal C\subset \tilde{\mathcal C}_n$. We have also $[J_f|\overline{i_d}]=[J_f|J_d]\setminus \left(\coprod_{J_d\ni b\neq \overline{i_d}}[J_f|b]\right)\in \tilde{\mathcal C}_n$ for any $J_f\neq \overline{J_f}$. Finally for any $b\in \mathcal A$ we get $[\overline{J_f}|b]=[b]\setminus\left(\coprod_{J_f\neq \overline{J_f}}[J_f|b]\right)$. But $[b]=\sigma^{-1}([b|])$, then $\zeta^n([b])=\sigma^{-|\zeta^n(b)|}\left(\zeta^{n}([b|])\right)$ and $[b|]=\coprod_{J_d}[b|J_d]\in \tilde{\mathcal C}_n$. Therefore 
 $\chi_{\zeta^n([b])}= \chi_{\zeta^{n}([b|])}\circ \sigma^{|\zeta^n(b)|}\in V_{\mathcal F_n^\mathcal C}$, thus $[b]\in \tilde{\mathcal{C}}_n$. We conclude that 
 $[\overline{J_f}|b]$ also belongs to $\tilde{\mathcal C}_n$.

\end{proof}

\begin{proof}[Proof of Theorem \ref{bofff}]
Let $\mathcal C$ be the collection given by Lemma \ref{ouch}. Then for any $A\in Q_n$,  the function $\chi_{A}$ belongs to $V_{\mathcal F_n^{\mathcal C}}$.  As the diameter of $Q_n$ goes to zero with $n$, the hypothesis of Lemma \ref{cyclic} are satisfied, so that $\M(X_\zeta, \sigma)\leq \sharp \mathcal C\leq K_\zeta$. 

When moreover the substitution has constant lenght $L$, then the induced system on $\zeta^n(X_\zeta)$ is just the power $\sigma^{L^n}$.  By Theorem \ref{mosh} the induced system is topologically conjugated to the substitution system, thus it is also uniquely ergodic. Let $\nu_n$ be the corresponding unique invariant measure. In particular for any clopen subset $C$ of $\zeta^n(X)$,  the sequence of continuous functions 
$$\frac{1}{K}\sum_{0\leq k\leq K}\chi_C\circ \sigma^{L^n}$$ is converging uniformly to $\nu_n(C)\chi_{\zeta^n(X)}$ when $K$ goes to infinity. 
When defining $\mathcal C$ in the previous lemma, we let all $[a|J_d]$ in $\mathcal C$ for $a\in \mathcal A$ and $J_d\in \mathcal A_d$. The elements of this form defines a partition and we may choose one such element $[\overline{a}|\overline{J_d}]$ such that $\zeta^n([\overline{a}|\overline{J_d}])$ has not full measure for the induced system. We can remove this element from $\mathcal C$. We let $\mathcal C'$ be  this new collection.   If $C'=[a|J_d]$ is an  element of $\mathcal C'$ with $\nu_n(\zeta^n(C'))>0$, then $\chi_{\zeta^n(X)}$ belongs to $V_{\chi_{\zeta^n (C')}}$ therefore $\chi_{\zeta^{n}([\overline{a}|\overline{J_d}])}$ lies in $V_{\mathcal F_n^{\mathcal C'}}$, i.e. $V_{\mathcal F_n^{\mathcal C}}=V_{\mathcal F_n^{\mathcal C'}}$.  Therefore $\M(X_\zeta, \sigma)\leq \sharp \mathcal  C'=\sharp \mathcal C-1<K_\zeta$.

\end{proof}

For the Thue-Morse substitution: $0\mapsto 01$ and $1\mapsto 10$, we get  $N_d=N_f=2$, thus $\M(X_\zeta, \sigma) < K_\zeta=4 $. 

\begin{ques}
    What is the topological multiplicity of the Thue-Morse substitution? 
\end{ques}

\section{Subshifts with linear growth complexity}
We consider a subshift $X\subset \mathcal A^\mathbb{Z}$ with letters  in a finite alphabet $\mathcal A$. For $x\in \mathcal A^\mathbb{Z}$ we denote by $x=(x_n)_{n\in \Z}$ for $x_n\in \mathcal{A}$. Let $\mathcal L_n(X)\subset \mathcal{A}^n$ be the finite words of $X$ of length $n$, i.e. $\mathcal L_n(X)=\{ x_k x_{k+1} \dots x_{k+n-1}: x\in X, k\in \Z \}$. The word complexity of $X$ is given by 
$$\forall n\in \mathbb N, \ \ p_X(n)=\sharp \mathcal L_n(X).$$
 We suppose that $X$ is aperiodic and has  linear growth, that is,  for some $k\in \mathbb N^*$ 
 \begin{equation}\label{eq:linear growth}
 	\liminf_n\frac{p_X(n)}{n}\le k.
 \end{equation}
 Boshernitzan \cite{boshernitzan1992condition} showed that such a subshift admits at most $k$ ergodic measures. By Theorem 5.5 in \cite{donoso2021interplay}  such  subshifts, when assumed to be moreover minimal,  have topological rank less than or equal to $(1+k\sharp\mathcal A^2)^{2(k+2)}$. We show in this section the following upper bound on the topological multiplicity. 
 
 \begin{thm}\label{mubo}
Any aperiodic subshift $X$ with $\liminf_{n\to \infty}\frac{p_X(n)}{n}\le k$ has topological multiplicity less than or equal to $2k$.
\end{thm}

One may wonder if the upper bound in Theorem \ref{mubo} is sharp. 
\begin{ques}Is an aperiodic subshift $X$ with $\liminf_{n\to \infty}\frac{p_X(n)}{n}=1$ topologically simple?
\end{ques}

In order to prove Theorem \ref{mubo}, we define some notations. Let $Q_n$ be the subset of $\mathcal L_n(X)$ given by words  $w$ such that there are several letters $a\in \mathcal A$ with  $wa\in \mathcal L_{n+1}(X)$. We also let $Q'_{n+1}$  be the $(n+1)$-words $wa$ as above. 
Clearly, we have
\begin{equation}\label{boc}
\sharp Q_{n} \leq p_X(n+1)-p_X(n) \text{ and } \sharp Q'_{n+1}=\sharp Q_n + p_X(n+1)-p_X(n).
\end{equation}
 
Through this section, we always assume the subshift is aperiodic and satisfies the linear growth \eqref{eq:linear growth}. 

\begin{lem}\label{lem:R}
The subset of integers 
$$
\mathcal N=\{n\in \N: p_X(n+1) < (k+1)(n+1) \text{ and } p_X(n+1)-p_X(n) \le k \}.
$$
is infinite.
\end{lem}	
For the sake of completeness, we reproduce the proof, which is contained in Theorem 2.2 of \cite{boshernitzan1984unique}.
\begin{proof}
	By \eqref{eq:linear growth}, we have 
	$$
	\liminf_n (p_X(n)-(k+1)n)=-\infty.
	$$
	It follows that
	$$
	\mathcal M =\left\{n\in \N: p_X(n+1)-(k+1)(n+1)\le \min\{0, \min_{1\le m\le n}\{p_X(m)-(k+1)m \} \}-1 \right\}
	$$
	is an infinite set. For any $n\in \mathcal M $, we have 
	$$
	p_X(n+1)-p_X(n)\le (k+1)(n+1)-(k+1)n-1=k.
	$$
	On the other hand, for any $n\in \mathcal M $, we get
	$$
 p_X(n+1)\le (k+1)(n+1)-1.
	$$
	This implies that $\mathcal M \subset \mathcal N$. Therefore, the set $\mathcal N$ is infinite.
\end{proof}

\begin{lem}[\cite{boshernitzan1984unique}, Lemma 4.1]\label{bound}
 For any $n\in \mathcal N$ and $m\ge (k+2)(n+1)$, any word $w\in \LL_m$ contains a subword in $Q_n$.
\end{lem}
For the sake of completeness we provide a  proof here.
\begin{proof}
	We prove it by contradiction. Assume to the contrary that all $(m - n + 1)$ $n$-subwords of $w$ do not belong to $Q_n$. That means that each of these $n$-blocks determines uniquely the next
	letter. Since $m - n + 1 \ge 2(k+1)n > p_X(n)$, at least one $n$-word appears more than one time as a subword of $w$. Therefore $X$ contains a periodic point. This  contradicts our assumption.
\end{proof}
 
Now we show that any cylinder of length less than $n$ can be decomposed as the cylinders of elements in $ Q'_{n+1}$ after translations. 
\begin{lem}\label{dense2}
Let $n\in \mathcal N$. Any cylinder $[w]$ with length of $w$ less than $n$ may be written uniquely as a finite disjoint union of sets of  the form $\sigma^{p}[q'_{n+1}]$ with $p\in \mathbb N$, $q'_{n+1}\in Q'_{n+1}$, such that  $\sigma^{t} [q'_{n+1}]\cap [q_n]=\emptyset$ for any $0<t<p$ and any $q_n\in Q_n$. 
\end{lem}
Remark  that by Lemma \ref{bound} the integers $p$ belongs to  $[0,(k+2)(n+1)]$.

\begin{proof}
Let $[w]$ be a cylinder associated to a word $w\in \mathcal L_l(X)$ with $l<n$. For $x\in [w]$, 
we let $K_x\in \mathbb Z$ be the largest integer $j$ less than  $l$ such that $x_{j-n+1}\cdots x_j$ belongs to  $Q_n$.  Then the word $w_{n+1}^x=x_{K_x-n+1}\cdots x_{K_x+1}$ belongs to $Q'_{n+1}$. Observe also that  by Lemma \ref{bound} we have $l-1-K_x\leq (k+2)(n+1)$, thus $n-1-K_x\leq (k+3)(n+1)$. 
 Let $W_{n+1}$ be the collection of these words $w_{n+1}^x$ over $x\in [w]$. By definition of $K_x$ and $Q_n$ the word $w_{n+1}^x$ completely determines the $l-1-K_x$ next letters, that is to say,
 $$
 [w_{n+1}^x]=[x_{K_x-n+1}\cdots x_{l-1}].
 $$
 For $x, y\in [w]$ the sets  $\sigma^{n-1-K_x}[w_{n+1}^x]$ and $\sigma^{n-1-K_y}[w_{n+1}^y]$ are either disjoint or equal. 
  As $x$ belongs to $[w]$ we have  $x\in \sigma^{n-1-K_x}[w_{n+1}^x]\subset [w]$ and finally 
$$[w]=\coprod_{w_{n+1}^x\in W_{n+1}} \sigma^{n-1-K_x}w^x_{n+1}.$$
We complete the proof.
\end{proof}

\begin{proof}[Proof of Theorem \ref{mubo}]By (\ref{boc}) and the definition of $\mathcal N$ 
we have for $n\in \mathcal N$ :
\begin{align*}
\sharp Q'_{n+1}&=\sharp Q_n + p_X(n+1)-p_X(n),\\
&\le 2(p_X(n+1)-p_X(n)),\\
&\le 2k.
\end{align*}
For $n\in \mathcal N$ we let $F_n =\{\chi_{[q'_{n+1}]}, \ q'_{n+1}\in Q'_{n+1}\}$. By Lemma \ref{dense2},  any cylinder $[w]$ with length less than $n$ is a finite disjoint union of  $\sigma^{p}[q'_{n+1}]$. In particular  $\chi_{[w]}$ lies in $V_{F_n}$. 
We may therefore apply  Lemma \ref{cyclic} to $(F_n)_{n\in \mathcal N}$ and we get $$\M(X,\sigma)\leq \sup_{n\in \mathcal N}\sharp Q'_{n+1}\le 2k.$$ 
\end{proof}

\subsection{Multiplicity of invariant measures}
It follows from Theorem \ref{mubo} and Lemma \ref{compa} that any ergodic measure has (ergodic) multiplicity bounded by $2k$. 	 In fact we may refine this result as follows: 
\begin{thm}\label{last}
Let  $X$  be an aperiodic subshift with $\liminf_{n\to \infty}\frac{p_X(n)}{n}\le k$. Then 
$$\sum_{\mu\in \mathcal M_e(X,\sigma)}\M(\mu)\leq 2k.$$ 
\end{thm}

In order to prove Theorem \ref{last}, we first recall some notations and then show two lemmas for general aperiodic subshifts. We have  learned from the referee that some parts of our proofs overlap with results in  \cite{creutz2023measuretheoretically}\cite{espinoza2023structure}. 
Let $(Y,\sigma)$ be an aperiodic subshift. For two finite words $w$ and $v$, we denote by $N(w|v)$ the number of times that $w$ appears as a subword of $v$. Also, We define $d(w|v)=N(w|v)/|v|$. For a generic point $x$ of a measure $\mu$, we have
$$
\lim_{n\to \infty} d(w|x_1^n) =\mu ([w]),
$$
where $x_1^n=x_1x_2\dots x_n$. For a finite word $v$, we denote by $v^{\otimes m}=\underbrace{vv \dots v}_{m \text{ times}}$.  For a finite word $w$, we denote by 
$$
\nu_w=\frac{1}{|w|}\sum_{k=0}^{|w|-1} \delta_{\sigma^k(\bar{w})},
$$
where $|w|$ is the length of $w$ and $\bar{w}\in \mathcal{A}^\Z$ is the periodization of $w$, i.e. $w^{\otimes \infty}$. \\

Let $w_n$ be a word of length $n$.  
Foy any $n$, we put 
$\ell_{n}=\ell(w_n):= \min \{1\le \ell<n: [w_{n}] \cap \sigma^\ell([w_{n}])\not=\emptyset \}$ and $L_{n}:=1+\sharp \{1\le \ell<n: [w_{n}] \cap \sigma^\ell([w_{n}])\not=\emptyset \}$, with the convention $\min\emptyset =n$. Let $v_{n}=v(w_n)$ be the first $\ell_{n}$-subword of $w_{n}$. It follows that $w_{n}= v_{n}^{\otimes K_{n}}\hat v_{n}$ with $\hat v_{n}\neq v_{n}$ being a prefix  of $v_{n}$. Then  $K_{n}=\lfloor n/\ell_{n} \rfloor\geq L_{n}$.  Observe  that for any $x\in X$, $p\geq \ell_{n}$ and any word $u$ of length less than $\ell_{n}$ we have 
\begin{align}\label{mdr}
N(u|x_1^p)&\geq  N(u|w_{n}) \frac{N(w_{n}|x_1^p)}{L_{n}},\\ \nonumber
&\geq N(u|v_{n}) K_{n} \frac{N(w_{n}|x_1^p)}{L_{n}},\\ \nonumber
&\geq  N(u|v_{n}) N(w_{n}|x_1^p).
\end{align}

In the next two lemmas we assume that 
\begin{itemize}
    \item the subshift $(Y, \sigma)$ is aperiodic;
    \item $w_n\in \mathcal{L}_n(Y)$ for $n\in \N$;
    \item $\nu_{w_n} \xrightarrow[n\in \mathcal{N}]{n\to +\infty}\nu \in \mathcal{M}_e(Y, \sigma)$, i.e. $\nu_{w_n}$ is weakly converging to an ergodic measure $\nu$ when $n$ goes to infinity along   a  subsequence $\mathcal{N}$.
\end{itemize}

\begin{lem}\label{int}
Under the above assumption, we have
$$\ell_{n}=|v_{n}|\xrightarrow[n\in \mathcal{N}]{n\to +\infty}+\infty$$
and $$\nu_{v_{n}}\xrightarrow[n\in \mathcal{N}]{n\to +\infty}\nu.$$
\end{lem}

\begin{proof} We argue by contradiction. Assume $(\ell_{n})_{n\in \N}$ has a bounded infinite subsequence $\mathcal{N}'$ of $\mathcal{N}$. Then there are finite words $v$ and $\hat v_n$ with $|\hat v_n|<|v|$ such that $w_{n}= v^{\otimes K_{n}}\hat v_n$ for $n\in \mathcal{N}''$ where $\mathcal{N}''$ is some infinite subsequence of $\mathcal{N}'$. Observe firstly that the length of $w_n$ goes to infinity. 
As a  consequence,  $K_n$ goes also to infinity as $n$ goes to infinite along $\mathcal{N}''$. But then $X$ should  contain  the periodic point $\overline{v}$ associated to $v$ which is a contradiction to the aperiodicity of $(Y, \sigma)$.  Therefore $\ell_{n}\xrightarrow[n\in \mathcal{N}]{n\to +\infty}+\infty$.

Let us check now that 
$\nu_{v_{n}}\xrightarrow[n\in \mathcal{N}]{n\to +\infty}\nu$. Let $\nu'=\lim_{k\to \infty}\nu_{v_{n_k}}$ be a weak limit of $(\nu_{v_{n}})_{n\in \mathcal{N}}$ with a subsequence $(n_k)_{k\in \N}$ of $\mathcal{N}$. For any word $u$ with $|u|<\ell_{n}$, by \eqref{mdr} we have  
$$
N(u|w_{n})\geq  N(u|v_{n})K_{n},
$$
and consequently
\begin{eqnarray*}
d(u|w_{n})\geq &d(u|v_{n}) \frac{K_{n}|v_{n}|}{(K_{n}+1)|v_{n}|},\\
\geq &\frac{1}{2}d(u|v_{n}). 
\end{eqnarray*}
For any cylinder $[u]$, By letting $n$ got infinity we get that
$$\nu([u])\geq \frac{1}{2}\nu'([u]).$$
It implies that $\nu-\frac{1}{2} \nu'$ is an $\sigma$-invariant measure. It follows from the  ergodicity of $\nu$ that $\nu=\nu'$. 
\end{proof}

\begin{lem}\label{keylem}For any ergodic measure $\mu\not=\nu$, we have  $$\lim_{n \in \mathcal N, n\to \infty} |v_{n}| \mu([w_{n}])=0.$$
\end{lem}
\begin{proof}
 Assume $\limsup_{n \in \mathcal N, n\to \infty} |v_{n}| \mu([w_{n}])>0.$ By passing to an infinite subsequence $\mathcal{N}'$ of $\mathcal{N}$ we have $\lim_{n\in \mathcal{N}', n\to \infty} |v_{n}| \mu([w_{n}])=b>0 $. By Lemma \ref{int} the sequence  $\nu_{v_{n}}$, $n\in \mathcal N'$, is converging to the measure $\nu$. 

 Let $x$ be a generic point of $\mu$. Then we have  for any $n$ 
\begin{equation}\label{eq11}
\lim_{p\to \infty} \frac{1}{p} \sum_{\ell=0}^{p-1} \chi_{[w_{n}]}(\sigma^\ell(x))= \lim_{p\to \infty} \frac{ N(w_{n}|x_1^p)}{p}=\mu([w_{n}]).
\end{equation}
In particular for any $n$ we can choose $P_n\in \N$ such that for $p\geq P_n$ 
\begin{equation}\label{eq33}
 \frac{ N(w_{n}|x_1^p)}{p}\geq \frac{\mu([w_{n}])}{2}.
\end{equation}

Pick an arbitrary cylinder $[u]$. by Lemma \ref{int} there exists an integer $N$ such that for $n>N$ we have
\begin{equation}\label{eq22}
N(u|v_{n})\ge \frac{1}{2} \nu([u]) |v_{n}|.
\end{equation}
It follows from (\ref{mdr}) that 
$$d(u|x^p_1)\geq \frac{1}{4}\nu([u])|v_{n}|\mu(w_{n}).$$
By letting $p$, then $n\in \mathcal N'$  go to infinity, we get for any cylinder $[u]$
$$\mu([u])\geq \frac{b}{4}\nu([u]).$$
It implies that $\mu-\frac{b}{4} \nu$ is an $\sigma$-invariant measure which is a contradiction to the ergodicity of $\mu$.
\end{proof}

We  recall now briefly  the proof of Boshernitzan that an aperiodic subshift of linear growth has finite many ergodic measures. Let $\mathcal{N}$ be the infinite set as in Lemma \ref{lem:R}. For any $n\in \mathcal{N}$, one can choose (not uniquely) an ordered $k$-tuple of $n$-words $K_n:=\{q_{n,1}, \dots, q_{n,k}\}$ which coincides with $Q_n$. By passing to a subsequence $\mathcal{N}'$ of $\mathcal{N}$, we can make each of the sequences of $\nu_{q_{n,i}}$ weakly convergences to some measures $\mu_i \in \mathcal{M}(X,T)$. Boshernitzan showed that 
$$
\mathcal{M}_e(X,T) \subset \{\mu_1, \mu_2, \dots, \mu_k \}.
$$
Since $\mu_i$ may coincide with the other $\mu_j$ for $j\not=i$, we define $I_i=\{1\le j\le k: \mu_j=\mu_i  \}$.\\

We will use the following complement  of Lemma \ref{dense2}.

\begin{lem} \label{dense3}In the decomposition  of a cylinder $[w]$ given by Lemma \ref{dense2}, for  any term  $\sigma^p[q'_{n+1}]$ with $|v(q'_{n+1})|<n+1$ we have $p\leq |v(q'_{n+1})|$. 
\end{lem}
\begin{proof}
We argue  by contradiction. To simplify the notations we write $v_{n}=v(q'_{n+1})$. Assume $|v_n|<n+1$ and $p> |v_{n}|$. By definition of $v_n$ we have  
$$
\emptyset\not= \sigma^{p}[q'_{n+1}] \cap \sigma^{p-|v_n|}[q'_{n+1}].
$$ But it follows from Lemma \ref{dense2} that  $\sigma^{p}[q'_{n+1}]$  does not intersect $\sigma^{l}\left(\bigcup_{q_n\in Q_n}[q_{n}]\right)$ for $0<   l< p $, therefore with $q_n\in Q_n$ being the prefix of $q'_{n+1}$ we get the contradiction 
$$ \sigma^{p}[q'_{n+1}] \cap \sigma^{p-|v_n|}[q'_{n+1}]\subset \sigma^{p}[q'_{n+1}] \cap \sigma^{p-|v_n|}[q_{n}]=\emptyset.$$
Thus we have $p\leq |v(q'_{n+1})|$.
\end{proof}

For a given $i$ we let $(q^l_{n,i})_{l\in \mathcal{Q}_{n,i}}$ be the elements of $Q'_{n+1}$ with prefix $q_{n,i}$, where $\mathcal{Q}_{n,i}$ is a subset of $\mathcal{A}$ for each $n\in \mathcal{N}', 1\le i\le k$. Note that $\nu_{q^l_{n,i}}$ is also converging to $\mu_i$ for any $l$ when $n$ goes to infinity along $\mathcal{N}'$. Finally we let  $v_{n,i}^l=v_n(q^l_{n,i})$ for each $n\in \mathcal{N}', 1\le i\le k$ and $l\in \mathcal{Q}_{n,i}$.

\begin{proof}[Proof of Theorem \ref{last}]
	Pick an arbitrary cylinder $[w]$.  Let 
	$$
	[w]=\coprod_{j,l} \coprod_{\in P_{n,j}^l} \sigma^{p}[q_{n,j}^l],
	$$  be the decomposition of $[w]$ given by Lemma \ref{dense2}.  Recall that  
  $P_{j,l}$ is a subset of $ [0,(k+2)(n+1)]$ for any $l$ and by Lemma \ref{dense3} we have also $P_{j,l}\subset [0,|v^{l}_{n,j}|-1]$ if $|v^{l}_{n,j}|<n+1$.  For each $n\in \mathcal{N}'$, we decompose $\{(j,l):1\le j\le k, l\in \mathcal{Q}_{n,j}\}$ into three set $J_{n,i}$, $J_{n,i}'$ and $J_{n,i}''$, where $J_{n,i}:=\{(j,l):j\in I_i\}$ , $J_{n,i}':=\{(j,l): j\notin I_i, |v^{l}_{n,j}|=n+1  \}$ and $J_{n,i}''$ is the rest. Then for $(j,l)\in J_{n,i}'$ we have
	\begin{equation}\label{eq666}
	\mu_i\left(\coprod_{p\in P_{n,j}^l} \sigma^{p}[q_{n,j}^l]\right)\le (k+2)(n+1) \mu_i([q_{n,j}^l]) = (k+2) |v^{l}_{n,j}|\mu_i([q_{n,j}^l]).
	\end{equation}
	On the other hand, for $(j,l)\in J_{n,i}''$, we have
	\begin{equation}\label{eq777}
	\mu_i\left(\coprod_{p\in P_{n,j}^l} \sigma^{p}[q_{n,j}^l]\right)\le |v^{l}_{n,j}|\mu_i([q_{n,j}^l]).
	\end{equation}
	By summing up \eqref{eq666} and \eqref{eq777}, we have 
	\begin{equation}
	\mu_i\left(\coprod_{(j,l)\in J_{n,i}'\cup J_{n,i}''} \coprod_{p\in P_{n,j}^l} \sigma^{p}[q_{n,j}^l]\right)\le 2k(k+2) \sum_{(j,l)\in J_{n,i}'\cup J_{n,i}''}  |v^{l}_{n,j}|\mu_i([q_{n,j}^l]).
	\end{equation}
	Combining this with Lemma \ref{keylem}, we obtain  	
	$$
	\lim_{n\to \infty}\mu_i\left(\coprod_{(j,l)\in J_{n,i}'\cup J_{n,i}''} \coprod_{p\in P_{n,j}^l} \sigma^{p}[q_{n,j}^l]\right)=0.
	$$
	Therefore we have
\begin{eqnarray*}
	\left\|\chi_{[w]}-\sum_{(n,j)\in J_{n,i}, \ p\in P_{n,j}^l} \chi_{[q_{n,j}^l]}\circ \sigma^{-p}\right\|^2_{L^2(\mu_i)}&
	= \left\|\chi_{[w]}-\chi_{\coprod_{(n,j)\in J_{n,i}, \ p\in P_{n,j}^l} \sigma^p[q_{n,j}^l]}\right\|^2_{L^2(\mu_i)},\\
&= \mu_i\left(\coprod_{(j,l)\in J_{n,i}'\cup J_{n,i}''} \coprod_{p\in P_{n,j}^l} \sigma^{p}[q_{n,j}^l]\right)\xrightarrow[n\in \mathcal{N}']{n\to +\infty}0.
	\end{eqnarray*}
Thus we can apply Lemma \ref{cyclic} in $L^2_0(\mu_i)$	with $F_n
=\{	\chi_{[q_{n,j}^l]} \  : \  (j,l)\in J_{n,i}\}$ to get 
$$\M(\mu_i)\le \liminf_{n\in \mathcal{N}', n\to \infty} \sharp J_{n,i}.$$
	By summing it up, we conclude that  
\begin{eqnarray*}	
	\sum_{\mu\in \mathcal M_e(X,\sigma)}\M(\mu)
	\leq &\sum_i
	\liminf_{n\in \mathcal{N}', n\to \infty} \sharp J_{n,i},\\
	\leq &\liminf_{n\in \mathcal{N}', n\to \infty}  \sharp Q'_{n+1},\\
	\leq &2k.	
	\end{eqnarray*}
\end{proof}	


\bibliographystyle{alpha}
\bibliography{universal_bib}

\end{document}